\documentclass[11pt,a4paper]{amsart}
\usepackage{amssymb,amsfonts,amsmath,mathrsfs}
\usepackage[left=2cm,right=2cm,top=2cm,height=250mm]{geometry}
\usepackage{color}
\usepackage[colorlinks=true,linkcolor=blue,citecolor=red, pdfborder={0 0 0}]{hyperref}

\numberwithin{equation}{section}
\newtheorem{theorem}{Theorem}[section]
\newtheorem{lemma}{Lemma}[section]
\newtheorem{Def}{Definition}[section]
\newtheorem{rem}{Remark}[section]
\newtheorem{prop}{Proposition}[section]
\newtheorem{cor}{Corollary}[section]

\def\d{\partial}
\def\R{\Bbb R}
\def\N{\Bbb N}
\def\Dx{\Delta_x}

\def\Dt{\partial_t}
\def\({\left(}
\def\){\right)}
\def\eb{\varepsilon}
\def\e{\epsilon}
\def\Cal{\mathcal}

\def\a{\alpha}
\def\E{\mathcal{E}}
\def\Ee{\mathcal{E}_{\{\eb,x_0\}}}
\def\Elr{\mathcal{E}_{\langle x_0\rangle}}
\def\A{\mathcal{A}}

\begin{document}
\title[Infinite energy solutions for fractional damped wave equation]{Infinite energy solutions for critical wave equation with fractional damping in unbounded domains}
\author[] {Anton Savostianov}

\begin{abstract}
This work is devoted to infinite-energy solutions of semi-linear wave equations in unbounded smooth domains of $\R^3$ with fractional damping of the form $(-\Dx+1)^\frac{1}{2}\Dt u$. The work extends previously known results for bounded domains in finite energy case. Furthermore, well-posedness and existence of locally-compact smooth attractors for the critical quintic non-linearity are obtained under less restrictive assumptions on non-linearity, relaxing some artificial technical conditions used before. This is achieved by virtue of new type Lyapunov functional that allows to establish extra space-time regularity of solutions of Strichartz type. 
\end{abstract}

\subjclass[2000]{35B40, 35B45, 35L70}

\keywords{damped wave equation, fractional damping, global attractor, smoothness, Strichartz estimates}
\thanks{The author is deeply grateful to Prof. Sergey Zelik for numerous fruitful discussions and insights concerning uniformly local spaces and infinite-energy solutions for PDEs. The author would like to thank the University of Surrey for hospitality and financial support, where essential part of this work was written.}
%\address{University of Surrey, Department of Mathematics, Guildford, GU2 7XH, United Kingdom.}
\address{Universit\'e de Cergy-Pontoise, CNRS UMR 8088, D\'epartement de Math\'ematiques, Cergy-Pontoise, F-95000, France.}
\email{anton.savostianov@u-cergy.fr}

\maketitle
\section{Introduction}
\label{s.intro}
In this work we study extra regularity and long-time behaviour of infinite-energy solutions in an unbounded smooth domain $\Omega\subset\R^3$ to the following problem 
\begin{equation}  
\begin{cases}
\label{eq intro}
\Dt^2u-\Dx u+\lambda_0 u+\gamma(-\Dx+1)^\theta\Dt u+f(u)=g,\ x\in\Omega,\\
u|_{\d\Omega}=0, u|_{t=0}=u_0,\ \Dt u|_{t=0}=u_1,
\end{cases}
\end{equation}
where constants $\gamma$ and $\lambda_0$ are strictly positive, $\theta=\frac{1}{2}$ and $u=u(t,x)$ is an unknown function. Further we assume that external force $g=g(x)$ belongs to $L^2_b(\Omega)$ and initial data $(u_0,u_1)\in \E_b$ defined as follows
\begin{equation}
\E_b=H^1_b(\Omega)\cap\{u|_{\d\Omega}=0\}\times L^2_b(\Omega),
\end{equation}
where spaces $L^2_b(\Omega)$ and $H^1_b(\Omega)$ are uniformly local analogues of $H^1(\Omega)$ and $L^2(\Omega)$ (see Section \ref{s.sp} for details). Non-linearity $f\in C^1(\R)$ is of critical growth and satisfies natural dissipative assumptions
\begin{align}
\label{f.growth}
&|f'(s)|\leq C(1+|s|^q),\\
\label{f(s)s>-M}
&f(s)s\geq -M,
\end{align} 
with $q=4$. Finally $-\Dx$ denotes usual laplacian and $(-\Dx+1)^\frac{1}{2}$ is operator $-\Dx+1$ to the power $\frac{1}{2}$ (see Section \ref{s.cest} for the details). 

Wave equations with fractional damping arise in situations when waves propagate through a lossy media, for example fractal rock layers, human tissues, different biomedical materials (see \cite{chen}, \cite{tree} and references therein). In contrast to usual damped wave equation ($\theta =0$) the fractional damping term $(-\Dx+1)^\theta \Dt u$ with $\theta>0$ allows to model processes in which which the higher frequencies decay faster than the lower frequencies. This can be easily seen in the linear homogeneous case on torus or in the whole space by means of Fourier series or Fourier transform. 

Let us give a brief review of known results for semi-linear wave equation with fractional damping in finite-energy case. 

It appears that even in linear case properties of wave equation with fractional damping significantly depend on parameter $\theta$. For instance, if $\theta =0$ we have finite speed of propagation property (see \cite{bk Sogge}). Also it is known that wave equation with fractional damping generates analytic semi-group if and only if $\theta\in[\frac{1}{2},1]$ (see \cite{tri1}, \cite{tri2}). When $\theta\in(0,1)$ wave equation with fractional damping possesses smoothing property similar to parabolic equations see \cite{KZ2009}, \cite{S ADE}, \cite{SZ2014}. In case $\theta=1$ there is instantaneous smoothing for $\Dt u$ and asymptotic smoothing for $u$ (see \cite{CV}).

In presence of non-linearity $f$ even the question of well-posedness of semi-linear wave equation with fractional damping becomes non-trivial. Indeed, in this case well-posedness of problem \eqref{eq intro} essentially depends not only on $\theta$ but also on growth exponent $q$ of the non-linearity (see \eqref{f.growth}). Let us first remind the known results when $\theta\in [0,1]\setminus \{\frac{1}{2}\}$. For a long time $q=4$ was considered as critical when $\theta\in[\frac{1}{2},1]$ (see \cite{carc3} and references therein). A great progress for the strongly damped wave equation ($\theta =1$) as well as the case $\theta\in[\frac{3}{4},1]$ was done in \cite{KZ2009}. In this case it was shown (see \cite{KZ2009}) that the mentioned problem is well-posed and possesses smooth attractor for {\it arbitrary} finite exponent $q\geq 0$ if $f\in C^1(\R)$ is such that
\begin{equation}
\label{f.growth2}
-C+a|s|^q\leq f'(s)\leq C(1+|s|^q),
\end{equation}
for some constants $a>0$ and $C\geq 0$.
In addition it was shown in \cite{KZ2009} that well-posedness and smooth attractor theory are valid if $\theta \in(\frac{1}{2},\frac{3}{4})$ and $q+2\leq\frac{6}{3-4\theta}$, that also improves exponent $4$.
Well-posedness and smooth attractor theory for the weakly damped wave equation ($\theta=0$) with critical non-linearity $q=4$ is considered in \cite{KSZ}. This result is based on substantial progress in Strichartz estimates for pure wave equation in bounded domains (see \cite{Sogge2009}, \cite{stri}) and trajectory attractor theory for weakly damped wave equation developed in \cite{ZelDCDS}. Let us remind that Strichartz estimates are space-time estimates where $L^p([0,T];L^r(\Omega))$-norm of the solution is controlled via energy norm of initial data and right-hand side for some admissible $p$ and $r$. For example when $\theta=0$ one may take $p=4$ and $r=12$. Furthermore, based on Strichartz type estimates, well-posedness and smooth attractor theory were obtained in sub-critical case $q\in[0,4)$ when $\theta\in(0,\frac{1}{2})$ (see \cite{S ADE}), whereas $q=3$ had been considered as critical before.   

Let us consider more carefully the case $\theta =\frac{1}{2}$. It appears that to get well-posedness and develop smooth attractor theory in this case one needs extra regularity similar to Strichartz-type estimates that does not follow directly from energy type inequality (the one obtained by multiplication of the equation \eqref{eq intro} by $\Dt u$). In \cite{SZ2014} it was shown that in fact any energy solution to problem \eqref{eq intro} belongs to $L^2([0,T];H^\frac{3}{2}(\Omega))$ under assumption
\begin{equation}
\label{f'>-c}
f'(s)\geq -C,
\end{equation}
and for arbitrary $q>0$ (restriction $q\leq 4$ is needed for the uniqueness but not for the extra regularity), if one considers periodic boundary conditions. The main observation used to obtain this extra regularity is the fact that actually one can multiply equation \eqref{eq intro} by $(-\Dx)^\frac{1}{2} u$. Then, in periodic case or in the whole space $\R^3$, similar to inequality
\begin{equation}
(f(u),-\Dx u)\geq -C\|u\|^2_{H^1(\Omega)},
\end{equation}
which is a consequence of \eqref{f'>-c} one can also get, again based on \eqref{f'>-c}, the following inequality (see \cite{SZ2014})
\begin{equation}
\label{fDx^1/2>-C1/2}
(f(u),(-\Dx)^\frac{1}{2} u)\geq -C\|u\|^2_{H^\frac{1}{2}(\Omega)}.
\end{equation}
And consequently the most problematic non-linear term in fact is not dangerous. Furthermore, in \cite{SZ2014} extra $L^2([0,T];H^\frac{3}{2}(\Omega))$-regularity was also obtained in case of Dirichlet boundary conditions. But, in case of Dirichlet boundary conditions, we do not know analogues of inequality \eqref{fDx^1/2>-C1/2}. Thus to obtain $L^2([0,T];H^\frac{3}{2}(\Omega))$-regularity in this case it was proposed to reduce the case of Dirichlet boundary conditions to the whole space $\R^3$ by making odd extension through the boundary. This in turn resulted in quite technical computations and artificial assumptions on non-linearity: oddness of $f$ and \eqref{f.growth2} instead of \eqref{f.growth} and \eqref{f'>-c}.  

Considering the case of unbounded domains it is natural to work with the class of so called infinite-energy solutions (the rigorous definition for our problem can be found in Section \ref{s.ex}, see also \cite{MZ Dafer 2008}). This is related to the fact that on the one hand if we considered our problem just in usual energy space $H^1_0(\Omega)\times L^2(\Omega)$, analogously to the case of bounded domains, then such choice of space would implicitly imply some decay condition in space variable at infinity for our solutions and therefore such solutions would be in some sense localised in space that looks restrictive. On the other hand space $L^\infty(\Omega)$ is also not convenient and usually is used in cases where maximum principle works. A good alternative was found in using uniformly local spaces (see \cite{MZ Dafer 2008} and references therein) which basic properties are collected in Section \ref{s.sp}. Infinite-energy solutions were constructed and studied for various problems of mathematical physics: reaction-diffusion systems (see \cite{Z sp.com}, \cite{Z entropy}), Cahn-Hilliard equation (\cite{JZ}), weakly damped wave equations (see \cite{F}, \cite{Z hypunbdd}), strongly damped wave equation (\cite{YSieswv}) and even Navier-Stokes equations (see \cite{AZ NStrp}, \cite{Z NSR2}, \cite{G.NS14} and therein) and dissipative Euler equation (see \cite{CZ} and therein). In \cite{F} infinite-energy solutions and corresponding attractor theory were constructed for autonomous weakly damped wave equation. The non-autonomous case as well as systematic study of Kolmogorov's $\eb$-entropy in both autonomous and non-autonomous case of weakly damped wave equation is presented in \cite{Z hypunbdd}. The attractor theory for the strongly damped wave equation (problem \eqref{eq intro} with $\theta=1$) and quintic non-linearity is developed in \cite{YSieswv}.    

The aim of this work is to extend this result to the case of infinite energy solutions in case of Dirichlet problem \eqref{eq intro} with $\theta=\frac{1}{2}$ in unbounded domains as well as relax assumptions \eqref{f.growth2}, oddness of $f$ and even \eqref{f'>-c} to less restrictive assumptions \eqref{f.growth} and \eqref{f(s)s>-M}. We notice, that when assumption \eqref{f'>-c} fails, but assumption \eqref{f(s)s>-M} is still valid, for example $\frac{\sin(u^5)}{u}$, it is natural to expect, that problem $\eqref{eq intro}$ does not possess finite-energy solutions if $\Omega$ is \emph{unbounded}. Therefore, the class of \emph{infinite-energy} solutions becomes the natural one in this case. In contrast to \cite{SZ2014}, the extra $L^2([0,T];\Cal H^{3/2}_{\{\eb,x_0\}}(\Omega))$-regularity, which is analogue of $L^2([0,T];H^{3/2}(\Omega))$-regularity in bounded domains, is achieved by one more preparatory step. Namely, we get $L^4([0,T];L^{12}_{\{\eb,x_0\}}(\Omega))$-norm control (weighted analogue of $L^4([0,T];L^{12}(\Omega))$) first. This is possible due to new type Lyapunov functional for problem \eqref{eq intro} that can be obtained by multiplication of \eqref{eq intro} by, roughly speaking, $u^3$ which is the main novelty of the work (see Theorems \ref{th adreg1.bdd}, \ref{th adreg1}). 
Obtained $L^4([0,T];L^{12}_{\{\eb,x_0\}}(\Omega))$-regularity implies that non-linear term $f(u)$ belongs to $L^1([0,T];L^2_{\{\eb,x_0\}}(\Omega))$ and hence we can multiply equation \eqref{eq intro} by, roughly speaking, $(-\Dx+1)^\frac{1}{2}u$ that gives desired $L^2([0,T];\Cal H^{3/2}_{\{\eb,x_0\}}(\Omega))$-regularity, but already without inequality \eqref{fDx^1/2>-C1/2}. Subsequent smoothing property (see Section \ref{s.sm}) and attractor theory (Section \ref{s.a}) for problem \eqref{eq intro} in unbounded domains can be then obtained in more or less standard way following the arguments from \cite{SZ2014}. 

The work is organised as follows. In Section \ref{s.sp} we collect basic definitions and properties of weighted and uniformly local spaces that are used throughout the paper. Commutator estimates involving fractional laplacian which are necessary for the work in weighed spaces are proven in Section \ref{s.cest}. Section \ref{s.ex} highlights the result on the existence of the infinite-energy solutions for problem \eqref{eq intro}. In Section \ref{s.exreg.fe} we explain the idea of the obtaining $L^4([0,T];L^{12}(\Omega))$-estimate for the problem \eqref{eq intro} in the simple case, when $\Omega$ is bounded. The key theorem of the work (Theorem \ref{th adreg1}) regarding $L^4([0,T];L^{12}_{\{\eb,x_0\}}(\Omega))$ extra regularity of infinite-energy solutions to problem \eqref{eq intro} is proven in Section \ref{s.exreg.ie}.
Based on the obtained extra regularity, uniqueness of solutions to problem \eqref{eq intro} is verified in Section \ref{s.u}. Smoothing property of infinite-energy solutions is established in Section \ref{s.sm}. Existence of locally compact smooth attractor is proven in Section \ref{s.a}. Finally, Appendix \ref{sec.Appendix} contains the proof of the dissipative estimate for the problem \eqref{eq intro} in the infinite-energy case where non-linearity $f$ satisfies \eqref{f.growth} and \eqref{f(s)s>-M} only. This proof is based on a variant of Gronwall-type inequality which can be of independent interest by itself.   
\section{Weighted and uniformly local spaces}
\label{s.sp}
In this section we introduce and recall basic properties of the family of weight functions and corresponding weighted Sobolev spaces as well as intimate relation between weighted Sobolev spaces and uniformly local spaces.

\begin{Def}
\label{def fsexpg}
Let $\mu >0$ be arbitrary. A function $\phi\in L^\infty_{loc}(\R^n)$ to be called a weight function of exponential growth $\mu$ iff $\phi(x)>0$ and there holds inequality
\begin{equation}
\label{2.1}
\phi(x+y)\leq C_\phi e^{\mu |y|}\phi(x),
\end{equation}
for every $x,y\in\R^n$.
\end{Def}
\begin{rem}
\label{rem 2.1}
One can easily check that if function $\phi$ is of exponential growth $\mu$ then so is the function $1/\phi$ with the same constant $C_\phi$. In other words \eqref{2.1} implies
\begin{equation}
\phi(x+y)\geq C_\phi^{-1}e^{-\mu|x|}\phi(y),
\end{equation}
for every $x,y \in\R^n$.
\end{rem}
\begin{prop} (see \cite{EZ2001})
Let $\phi_1$ and $\phi_2$ be weight functions of exponential growth $\mu_1$ and $\mu_2$ respectively. Then

(i) $\alpha\phi_1+\beta\phi_2$, $\max\{\phi_1 ,\phi_2 \}$ and $\min\{\phi_1,\phi_2\}$ are weight functions of exponential growth $\max\{\mu_1,\mu_2\}$ for every $\alpha$, $\beta>0$;

(ii) $\phi_1\cdot\phi_2$ and $\phi_1\cdot \phi_2^{-1}$ are weight functions of exponential growth $\mu_1+\mu_2$;

(iii) $\phi_1^\a$ is a weight function of exponential growth $|\a|\mu_1$.
\end{prop}

The main examples of weight functions of exponential growth are $e^{-\eb|x-x_0|}$, its smooth analogue $e^{-\eb\sqrt{1+|x-x_0|^2}}$ and $(1+|x-x_0|^2)^\a$ where $\eb$ and $\a$ belong to $\R$. It is easy to see that the first two examples are functions of exponential growth $|\eb|$ and the last one is the weight function of exponential growth $\mu$ for arbitrary $\mu>0$. For us the weight $\phi_{\eb,x_0}(x)=e^{-\eb\sqrt{1+|x-x_0|^2}}$ with small enough $\e>0$ will be of particular interest.

\begin{Def}\label{def Lp phi}
Let $\Omega$ be some unbounded domain in $\R^n$, and let $\phi$ be a weight function of exponential growth $\mu$. Then the space $L^p_\phi$ is defined as follows
\begin{equation}
L^p_\phi(\Omega)=\left\{u\in \Cal D'(\Omega): \|u,\Omega\|^p_{\phi,0,p}\equiv \int_\Omega\phi^p(x)|u(x)|^pdx<\infty\right\}.
\end{equation} 
\end{Def}

Analogously we define $W^{l,p}_{\phi}(\Omega)$, $l\in \N$, as the space of those distributions which derivatives up to the order $l$ belong to $L^p_\phi(\Omega)$. In particular we use notation $W^{l,p}_{\{\eb,x_0\}}(\Omega)$  if the corresponding weight is $e^{-\eb\sqrt{1+|x-x_0|^2}}$. Also we denote by $\|\cdot\|_{\{\eb,x_0\},l,p}$ the norm in $W^{l,p}_{\{\eb,x_0\}}(\Omega)$. 

We also define so called uniformly local Sobolev spaces, which are
also of special interest to us,
\begin{Def}
\label{def Wlpb}
Let $\Omega\subset\R^n$ be an unbounded domain. Then 
\begin{equation}
W^{l,p}_b(\Omega)=\left\{u\in \Cal D'(\Omega): \|u,\Omega\|_{b,l,p}=\sup_{x_0\in\R^n}\|u,\Omega\cap B^1_{x_0}\|_{l,p}<\infty\right\}.
\end{equation}
\end{Def}
Here and below $B^R_{x_0}$ stands for the ball in $\R^n$ of radius $R$  with center at $x_0$ and $\|u,V\|_{l,p}$ means $\|u\|_{W^{l,p}(V)}$ for a domain $V\subset \R^n$.
 
\begin{theorem}\label{th wsp1} (see \cite{EZ2001})
Let $u\in L^p_\phi(\Omega)$ where $\phi$ is a weight function of exponential growth $\mu>0$. Then for any $1\leq q\leq \infty$ the following estimate is valid:
\begin{equation}
\label{2.2}
\(\int_\Omega\phi(x_0)^{pq}\(\int_{\Omega}e^{-p\eb|x-x_0|}|u(x)|^p\)^qdx_0\)^\frac{1}{q}\leq C\int_\Omega\phi^p(x)|u(x)|^pdx,
\end{equation}
for every $\eb>\mu$, where the constant $C$ depends only on $\eb$, $\mu$ and $C_\phi$ from \eqref{2.1} (and is independent of $\Omega$).
\end{theorem}

For further properties of weighted Sobolev spaces defined above we need some restrictions related to smoothness of the boundary of the domain and in particular its smoothness "at infinity". First, we assume there exists such a number $R_0>0$ that for every $x\in\Omega$ there is a smooth domain $V_x\subset\Omega$ such that
\begin{equation}
\label{2.6}
B^{R_0}_x\cap\Omega\subset V_x\subset B^{R_0+1}_x\cap \Omega.
\end{equation}
This assumption allows to avoid, for example, such pathologies as infinitely thin cuts going to infinity and boundaries that behave like $\sin(x^2)$. Second, we assume that there exists a diffeomorphism $\theta_x: B^1_0 \to V_x$ such that
\begin{equation}
\label{2.7}
\|\theta_x\|_{C^N}+\|\theta_x^{-1}\|_{C^N}\leq K,
\end{equation}
uniformly with respect to $x\in\Omega$ for sufficiently large $N$. This assumption, for example, allows to exclude infinitely thin rays. In general these assumptions are standard and for bounded domains this is equivalent to the fact that the boundary is smooth enough manifold.  
The above mentioned assumptions supposed to be valid throughout the work with $R_0=2$ (for simplicity). 
\begin{theorem}(see \cite{EZ2001})\label{th wsp2}
Let domain $\Omega$ satisfies conditions \eqref{2.6} and \eqref{2.7}, $\phi$ be a function of exponential growth and let $R>0$ be a fixed number. Then for any $p\in[1;\infty)$ the following estimates are valid:
\begin{equation}
C_1\int_\Omega \phi^p(x)|u(x)|^pdx\leq \int_{\Omega}\phi^p(x_0)\int_{\Omega\cap B^R_{x_0}}|u(x)|^pdxdx_0\leq C_2\int_{\Omega}\phi^p(x)|u(x)|^pdx.
\end{equation}
\end{theorem}
Applying Theorem \ref{th wsp2} to derivatives of $u$ we get:
\begin{cor}
\label{cor wsp3} (see \cite{EZ2001})
Let $l\in\Bbb{N}$, $p\in[1;\infty)$ and the domain $\Omega$ satisfies \eqref{2.6}, \eqref{2.7}. Then an equivalent norm in $W^{l,p}_{\phi}(\Omega)$ is given by the following expression:
\begin{equation}
\label{2.8}
\|u,\Omega\|_{\phi,l,p}=\(\int_\Omega\phi^p(x_0)\|u,\Omega\cap B^R_{x_0}\|^p_{l,p}dx_0\)^\frac{1}{p}.
\end{equation}
In particular we obtain that norms \eqref{2.8} are equivalent for different $R>0$.
\end{cor}
Moreover, using Theorem \ref{th wsp2} we get equivalent norms for uniformly local Sobolev spaces (see \cite{Z entropy}):
\begin{cor}\label{cor wsp4}
Let domain $\Omega$ satisfy assumptions \eqref{2.6}, \eqref{2.7} and $u\in W^{l,p}_b(\Omega)$. Then for any $\eb>0$ we have
\begin{equation}
C_1\|u,\Omega\|^p_{b,l,p}\leq\sup_{x_0\in\Omega}\left\{\int_{x\in\Omega}e^{-p\eb|x-x_0|}\|u,\Omega\cap B^1_x\|^p_{l,p}dx\right\}\leq C_2\|u,\Omega\|^p_{b,l,p}.
\end{equation} 
\end{cor}
We see that representation \eqref{2.8} reduces weighted Sobolev norm to Sobolev norm on bounded domains. Particularly, this gives the benefit of using standard Sobolev embeddings theorems for bounded domains (see \cite{Zstr2007}). Moreover, in analogy to \eqref{2.8} we are able to define fractional weighted Sobolev spaces that will play an important role in the sequel
\begin{Def} (see \cite{EZ2001})
Let $s\geq 0$, $p\in[1;\infty)$ and $R>0$ be fixed numbers and domain $\Omega$ enjoys \eqref{2.6} and \eqref{2.7}. The $W^{s,p}_\phi(\Omega)$ is defined by
\begin{equation}
\label{2.9}
W^{s,p}_\phi(\Omega)=\left\{u\in D'(\Omega): \int_\Omega\phi^p(x_0)\|u,\Omega\cap B^R_{x_0}\|^p_{s,p}dx_0<\infty\right\},
\end{equation}
so the corresponding norm is given by \eqref{2.8} with $s$ instead of $l$. 
\end{Def}
Let us remind that Sobolev norm in $W^{s,p}(V)$ for a domain $V$ with non-integer positive $s$ can be defined by
\begin{equation}
\label{2.10}
\|u,V\|^p_{s,p}=\|u,V\|^p_{[s],p}+\sum_{|\alpha|=[s]}\int_{x\in V}\int_{y\in V}\frac{|\d^\alpha u(x)-\d^\alpha u(y)|^p}{|x-y|^{n+\{s\}p}}dxdy,
\end{equation} 
where $[s]$ and $\{s\}$ denote integer and fractional part of $s$ respectively. It is not difficult to check that norms defined by \eqref{2.9} are equivalent for different $R>0$ as well as \eqref{2.9} indeed gives usual
norm for $W^{s,p}(\Omega)$ if we take $\phi\equiv 1$ and $\Omega$ is a smooth bounded domain (see \cite{EZ2001}). Hence the above definition is natural. Analogously to the integer case we use notations $W^{s,p}_{\{\eb,x_0\}}(\Omega)$ and  $W^{s,p}_b(\Omega)$.
 
\section{Commutator estimates for the fractional laplacian}
\label{s.cest}
This section is devoted to commutator estimates for the fractional laplacian involving functions of exponential growth.

Let $\eb\in \R$ and $x_0\in\Omega$ be fixed and $\Dx$ denotes the usual laplacian. The operator 
$$-\Dx+1:L^2_{\{\eb,x_0\}}(\Omega)\to L^2_{\{\eb,x_0\}}(\Omega)$$ 
endowed with Dirichlet boundary conditions can be considered as an unbounded operator in $L^2_{\{\eb,x_0\}}(\Omega)$ with the domain of definition $\Cal H^2_{\{\eb,x_0\}}:=\Cal D(-\Dx+1)$. One of the possible rigorous definitions of $(-\Dx+1)^\theta$ is as follows (see \cite{yos}, Chapter IX, Section 11; see also \cite{LM,triebel} for the descriptions of the domains of definitions)
\begin{Def}
\label{Def fp}
The operator $(-\Dx+1)^\theta$, for $\theta\in(0,1)$ and $\eb\in \R$, is understood as unbounded operator in $L^2_{\{\eb,x_0\}}(\Omega)$ defined by the formula
\begin{equation}
\label{fp formula}
(-\Dx+1)^\theta u=\frac{1}{\Gamma(-\theta)}\int_0^\infty\lambda^{-\theta -1}\(e^{-(-\Dx +1)\lambda}-1\)ud\lambda,\qquad \theta\in (0,1),
\end{equation}
where $\Gamma$ is the standard gamma function, with the domain of definition $\Cal H^{2\theta}_{\{\eb,x_0\}}(\Omega):=\Cal D((-\Dx+1)^\theta)$, which in terms of Sobolev spaces can be described as follows
\begin{align}
\label{D(Dx)<1/4}
&\Cal H^{2\theta}_{\{\eb,x_0\}}(\Omega)=H^{2\theta}_{\{\eb,x_0\}}(\Omega),\ \rm{when}\ \theta\in\left(0,\frac{1}{4}\right);\\
\label{D(Dx)1/4}
&\Cal H^\frac{1}{2}_{\{\eb,x_0\}}(\Omega)=H^\frac{1}{2}_{\{\eb,x_0\}}(\Omega)\cap \left\{u:\ \int_\Omega\frac{\phi^2_{\eb,x_0}(x)|u(x)|^ 2}{d(x)}dx<\infty\right\},\ {\rm where}\ d(x)={\rm dist}(x,\partial\Omega);\\
\label{D(Dx)>1/4}
&\Cal H^{2\theta}_{\{\eb,x_0\}}(\Omega)=H^{2\theta}_{\{\eb,x_0\}}(\Omega)\cap \left\{u:u|_{\partial\Omega}=0\right\},\ {\rm when}\ \theta\in\bigg(\bigg.\frac{1}{4},1\bigg.\bigg].
\end{align}
\end{Def} 
One can check that this definition coincides with the usual definition of the laplacian when $\Omega$ is a smooth bounded domain or the whole space. We also recall that $\Cal H^{2\theta}_{\{\eb,x_0\}}(\Omega)$ with $\theta\in [0,1]$ and $\eb\in\R$ are Hilbert spaces and $\Cal H^{-2\theta}_{\{-\eb,x_0\}}(\Omega):=\left(\Cal H^{2\theta}_{\{\eb,x_0\}}(\Omega)\right)^*$ can be understood as dual to $\Cal H^{2\theta}_{\{\eb,x_0\}}(\Omega)$.

The next commutator estimate will be used frequently throughout the work. 
\begin{prop}
\label{prop comest1}
Let $\theta\in(0,\frac{1}{2})$, $\eb\in(0,\eb_0]$ with small enough $\eb_0$ and $u\in L^2_{\{\eb,x_0\}}(\Omega)$. Then the following commutator estimate  holds true
\begin{equation}
\|(-\Dx+1)^\theta(\phi_{\eb,x_0} u)-\phi_{\eb,x_0}(-\Dx+1)^\theta u\|_{L^2(\Omega)}\leq C\eb^\frac{1}{2}2^{\frac{1}{2}-\theta}\frac{\Gamma\left(\frac{1}{2}-\theta\right)}{|\Gamma(-\theta)|}\|\phi_{\eb,x_0}u\|_{L^2(\Omega)},
\end{equation}
for some absolute constant $C$, where  $\phi_{\eb,x_0}(x)=e^{-\eb\sqrt{1+|x-x_0|^2}}$.
\end{prop}
\begin{proof}
Writing down the difference under consideration using \eqref{fp formula} it is clear that we should get a control of the quantity
\begin{equation}
\|e^{-(-\Dx +1)\lambda}(\phi_{\eb,x_0}u)-\phi_{\eb,x_0}e^{-(-\Dx+1)\lambda}u\|_{L^2(\Omega)}.
\end{equation}
Let us denote $v(\lambda):=e^{-(-\Dx +1)\lambda}(\phi_{\eb,x_0}u)$ and $w(\lambda):=e^{-(-\Dx +1)\lambda}u$, that is $v$ and $w$ satisfy the equations
\begin{equation}
\label{3.1}
\begin{cases}
\d_\lambda v+(-\Dx+1)v=0,\ x\in\Omega\\
v|_{\d\Omega}=0,\ v(0)=\phi_{\eb,x_0}u,
\end{cases},\quad
\begin{cases}
\d_\lambda w+(-\Dx+1)w=0,\ x\in\Omega,\\
w|_{\d\Omega}=0,\ w(0)=u.
\end{cases}
\end{equation}
Therefore the quantity we need to control is
\begin{equation}
\|v(\lambda)-\phi_{\eb,x_0}w(\lambda)\|_{L^2(\Omega)}.
\end{equation} 
From \eqref{3.1} we see that $v-\phi_{\eb,x_0}w$ solves the problem
\begin{equation}
\begin{cases}
\label{3.2}
\d_{\lambda}(v-\phi_{\eb,x_0}w)+(-\Dx+1)(v-\phi_{\eb,x_0}w)=-2\nabla\phi_{\eb,x_0}\nabla w-w\Dx\phi_{\eb,x_0},\\
(v-\phi_{\eb,x_0}w)|_{\d\Omega}=0,\ (v-\phi_{\eb,x_0}w)|_{\lambda=0}=0.
\end{cases}
\end{equation}
Multiplying equation from \eqref{3.2} by $v-\phi_{\eb,x_0}w$ we find 
\begin{multline}
\label{3.3}
\frac{1}{2}\frac{d}{d\lambda}\|v-\phi_{\eb,x_0}w\|^2_{L^2(\Omega)}+\|\nabla(v-\phi_{\eb,x_0}w)\|^2_{L^2(\Omega)}+\|v-\phi_{\eb,x_0}w\|^2_{L^2(\Omega)}= \\
-2(\nabla\phi_{\eb,x_0}\nabla w,v-\phi_{\eb,x_0}w)-(w\Dx\phi_{\eb,x_0},v-\phi_{\eb,x_0}w)=\\
2(w,\nabla\phi_{\eb,x_0}\nabla(v-\phi_{\eb,x_0}w))+(w\Dx\phi_{\eb,x_0},v-\phi_{\eb,x_0}w):=A_1+A_2.
\end{multline}
By simple calculations one finds $|\nabla \phi_{\eb,x_0}|\leq \eb\phi_{\eb,x_0}$ that immediately gives the estimate for $A_1$
\begin{equation}
\label{3.4}
|A_1|\leq 2\eb\|\phi_{\eb,x_0}w\|_{L^2(\Omega)}\|\nabla(v-\phi_{\eb,x_0}w)\|_{L^2(\Omega)}\leq\eb\|\phi_{\eb,x_0}w\|^2_{L^2(\Omega)}+\eb\|\nabla(v-\phi_{\eb,x_0}w)\|^2_{L^2(\Omega)}.
\end{equation}
Analogously we see that $|\Dx\phi_{\eb,x_0}|\leq C\eb\phi_{\eb,x_0}$ with some absolute constant $C$ that implies $A_2$ estimate
\begin{equation}
\label{3.5}
|A_2|\leq C\eb \|\phi_{\eb,x_0}w\|_{L^2(\Omega)}\|v-\phi_{\eb,x_0}\|_{L^2(\Omega)}\leq C\eb\|\phi_{\eb,x_0}w\|^2_{L^2(\Omega)}+C\eb\|v-\phi_{\eb,x_0}\|^2_{L^2(\Omega)}.
\end{equation}
Since $w$ solves \eqref{3.1} it is not difficult to see, multiplying the corresponding equation by $\phi^2_{\eb,x_0}w$, that $w$ admits the estimate 
\begin{equation}
\label{3.6}
\|\phi_{\eb,x_0}w(\lambda)\|^2_{L^2(\Omega)}\leq e^{-\lambda}\|\phi_{\eb,x_0}u\|^2_{L^2(\Omega)},\quad \lambda\geq 0,
\end{equation}
as long as $\eb\leq \eb_0$.
Combining \eqref{3.3}-\eqref{3.6} we arrive at
\begin{equation}
\frac{d}{d\lambda}\|v-\phi_{\eb,x_0}w\|^2_{L^2(\Omega)}+\|v-\phi_{\eb,x_0}w\|^2_{L^2(\Omega)}\leq
 \eb Ce^{-\lambda}\|\phi_{\eb,x_0}u\|^2_{L^2(\Omega)}
\end{equation}
for some absolute constant $C\geq 0$, as long as $\eb\leq\eb_0$. Hence multiplying the above equation by $e^{\lambda}$ one finds
\begin{equation}
\label{3.7}
\|v-\phi_{\eb,x_0}w\|_{L^2(\Omega)}\leq \sqrt{C}\eb^\frac{1}{2}\lambda^\frac{1}{2}e^{-\frac{\lambda}{2}}\|\phi_{\eb,x_0}u\|_{L^2(\Omega)}.
\end{equation}
Finally, using \eqref{fp formula} and \eqref{3.7}, we are ready to get commutator estimate
\begin{multline}
\label{3.7.1}
\|(-\Dx+1)^\theta(\phi_{\eb,x_0} u)-\phi_{\eb,x_0}(-\Dx+1)^\theta u\|_{L^2(\Omega)}\leq\\
\frac{1}{|\Gamma(-\theta)|}\int_0^\infty\lambda^{-1-\theta}\|v(\lambda)-\phi_{\eb,x_0}w(\lambda)\|_{L^2(\Omega)}d\lambda\leq\\
\sqrt{C}\eb^\frac{1}{2}\frac{1}{|\Gamma(-\theta)|}\int_0^\infty\lambda^{-\frac{1}{2}-\theta}e^{-\frac{\lambda}{2}}d\lambda\,\|\phi_{\eb,x_0}u\|_{L^2(\Omega)}=\left|s=\frac{\lambda}{2}\right|=\\
\sqrt{C}\eb^\frac{1}{2}2^{\frac{1}{2}-\theta}\frac{1}{|\Gamma(-\theta)|}\int_0^\infty s^{\left(\frac{1}{2}-\theta\right)-1}e^{-s}ds\,\|\phi_{\eb,x_0}u\|_{L^2(\Omega)}=\sqrt{C}\eb^\frac{1}{2}2^{\frac{1}{2}-\theta}\frac{\Gamma\left(\frac{1}{2}-\theta\right)}{|\Gamma(-\theta)|}\|\phi_{\eb,x_0}u\|_{L^2(\Omega)},
\end{multline}
that completes the proof.
\end{proof}
 
\begin{cor}
\label{cor equiv n1}
Let $\eb\in(0,\eb_0]$ be small enough, $\theta\in (0,\frac{1}{2})$ and $u\in \Cal H^{2\theta}_{\{\eb,x_0\}}(\Omega)$ then the following norms are equivalent
\begin{multline}
C_1(\theta)\|(-\Dx+1)^\theta (\phi_{\eb,x_0}u)\|_{L^2(\Omega)}\leq\\
\|\phi_{\eb,x_0}(-\Dx+1)^\theta u\|_{L^2(\Omega)}\leq\\
C_2(\theta)\|(-\Dx+1)^\theta (\phi_{\eb,x_0}u)\|_{L^2(\Omega)},
\end{multline} 
where constants $C_i(\theta)$ are independent of $\eb$.
\end{cor}
Below we prefer to use the norm given by Corollary \ref{cor equiv n1}.
 
As it is seen from Proposition \ref{prop comest1} the case $\theta=\frac{1}{2}$ is in a sense critical and for this reason is delicate. However for our purposes we will need only $\theta=\frac{1}{4}$. On the other hand to establish Corollary \ref{cor equiv n1} for the case $\theta\in[\frac{1}{2},1)$ we may allow a weaker version of Proposition \ref{prop comest2} 

%%%%%%%%%%%%%%%%%%%%%%%%%%%%%%%%%%%%%%%%%%%%%%%%%%%%%%%%%%%%%%%%%%%%%

\begin{prop}
\label{prop comest2}
Let $s\in(0,1)$, $\theta\in (0,\frac{1+s}{2})$ and $u\in \Cal H^s_{\{\eb,x_0\}}(\Omega)$. Then for any small enough $\eb\in(0,\eb_0]$ the following commutator estimate holds
\begin{multline}
\|(-\Dx+1)^\theta(\phi_{\eb,x_0} u)-\phi_{\eb,x_0}(-\Dx+1)^\theta u\|_{L^2(\Omega)}\leq\\
 C_s\eb^\frac{1+s}{2}2^{\frac{1+s}{2}-\theta}\frac{\Gamma\left(\frac{1+s}{2}-\theta\right)}{|\Gamma(-\theta)|}\|\phi_{\eb,x_0}(-\Dx+1)^\frac{s}{2}u\|_{L^2(\Omega)},
\end{multline}
with some constant $C_s> 0$, that depends only on $s$.
\end{prop}
\begin{proof}
The first part is the same as in Proposition \ref{3.3}. So we consider the difference $v-\phi_{\eb,x_0}w$, with $v$ and $w$ given by \eqref{3.1}, which satisfies the identity \eqref{3.3}. Now to estimate $A_1$ we proceed slightly different
\begin{multline}
\label{3.8}
|A_1|=2|\sum_{k=1}^3(w\d_k\phi_{\eb,x_0},\d_k(v-\phi_{\eb,x_0}w))|=2\eb|\sum_{k=1}^3(w\phi_{\eb,x_0}a_k(x),\d_k(v-\phi_{\eb,x_0}w))|\leq\\
2\eb\sum_{k=1}^3\|(-\Dx+1)^\frac{s}{2}\left(\phi_{\eb,x_0}a_k(x)w\right)\|_{L^2(\Omega)}\|(-\Dx+1)^{-\frac{s}{2}}(\d_k(v-\phi_{\eb,x_0}w))\|_{L^2(\Omega)},
\end{multline}
where $a_k(x)=\frac{x_k-(x_0)_k}{\sqrt{1+|x-x_0|^2}}$. Now, since $\frac{s}{2}\in\left(0,\frac{1}{2}\right)$, we can complete the above estimate by using Proposition \ref{prop comest1} with $\psi^k_{\e,x_0}(x)=a_k(x)\phi_{\eb,x_0}(x)\in C^\infty(\Omega)$ instead of $\phi_{\eb,x_0}$. Indeed, clearly that
$|\nabla \psi^k_{\eb,x_0}|\leq C\phi_{\eb,x_0}$ and $|\Dx \psi^k_{\eb,x_0}|\leq C \phi_{\eb,x_0}$ for some absolute constant $C$ when $\eb\in(0,\eb_0]$,
thus we derive
\begin{multline}
\label{3.85}
|A_1|\leq\\
2C_s\eb\sum_{k=1}^3\left(\|\phi_{\eb,x_0}(-\Dx+1)^\frac{s}{2}w\|_{L^2(\Omega)}+\|\phi_{\eb,x_0}w\|_{L^2(\Omega)}\right)\|v-\phi_{\eb,x_0}w\|^s_{L^2(\Omega)}\|v-\phi_{\eb,x_0}w\|^{1-s}_{H^1(\Omega)}\leq\\
2C_s\eb\|\phi_{\eb,x_0}(-\Dx+1)^\frac{s}{2}w\|_{L^2(\Omega)}\|v-\phi_{\eb,x_0}w\|^s_{L^2(\Omega)}\|v-\phi_{\eb,x_0}w\|^{1-s}_{H^1(\Omega)},
\end{multline}
with some constant $C_s$ depending only on $s$. Then applying Young's inequality on the right hand side of \eqref{3.85} with exponents $\frac{2}{1-s}$ and $\frac{2}{1+s}$ we deduce
\begin{equation}
\label{3.9}
|A_1|\leq 2C_s\eb\(\|\phi_{\eb,x_0}(-\Dx+1)^\frac{s}{2}w\|^\frac{2}{1+s}_{L^2(\Omega)}\|v-\phi_{\eb,x_0}w\|^\frac{2s}{1+s}_{L^2(\Omega)}+\|v-\phi_{\eb,x_0}w\|^2_{H^1(\Omega)}\).
\end{equation}
$A_2$ term is easier and can be estimated as follows
\begin{multline}
\label{3.10}
|A_2|\leq C\eb\|\phi_{\eb,x_0}w\|_{L^2(\Omega)}\|v-\phi_{\eb,x_0}\|_{L^2(\Omega)}\leq\\
 C_s\eb\|\phi_{\eb,x_0}(-\Dx+1)^\frac{s}{2}w\|_{L^2(\Omega)}\|v-\phi_{\eb,x_0}w\|^s_{L^2(\Omega)}\|v-\phi_{\eb,x_0}w\|^{1-s}_{H^1(\Omega)}\leq\\
C_s\eb\(\|\phi_{\eb,x_0}(-\Dx+1)^\frac{s}{2}w\|^\frac{2}{1+s}_{L^2(\Omega)}\|v-\phi_{\eb,x_0}w\|^\frac{2s}{1+s}_{L^2(\Omega)}+\|v-\phi_{\eb,x_0}w\|^2_{H^1(\Omega)}\),
\end{multline}
with some constant $C_s$ depending only on $s$.
Thus combining \eqref{3.3}, \eqref{3.9}, \eqref{3.10}, taking $\eb\leq\eb_0$ small enough and taking into account that 
\begin{multline}
\frac{d}{d\lambda}\|v-\phi_{\eb,x_0}w\|^2_{L^2(\Omega)}=\frac{d}{d\lambda}\left(\|v-\phi_{\eb,x_0}w\|_{L^2(\Omega)}^\frac{2}{1+s}\right)^{1+s}=\\
(1+s)\|v-\phi_{\eb,x_0}w\|_{L^2(\Omega)}^\frac{2s}{1+s}\frac{d}{d\lambda}\|v-\phi_{\eb,x_0}w\|_{L^2(\Omega)}^\frac{2}{1+s}
\end{multline}
we derive the estimate
\begin{equation}
\label{3.11}
\frac{d}{d\lambda}\|v-\phi_{\eb,x_0}w\|^\frac{2}{1+s}_{L^2(\Omega)}+\frac{1}{1+s}\|v-\phi_{\eb,x_0}w\|^\frac{2}{1+s}_{L^2(\Omega)}\leq C_s\eb\|\phi_{\eb,x_0}(-\Dx+1)^\frac{s}{2}w\|^\frac{2}{1+s}_{L^2(\Omega)},
\end{equation}
as long as $\eb\in(0,\eb_0]$, with some constant $C_s$ depending only on $s$. Applying operator $(-\Dx+1)^\frac{s}{2}$ to the equation for $w$ \eqref{3.1} and repeating the arguments of \eqref{3.6} for $\tilde w=(-\Dx+1)^\frac{s}{2}w$ we obtain 
\begin{equation}
\label{3.13}
\|\phi_{\eb,x_0}(-\Dx+1)^\frac{s}{2}w(\lambda)\|^2_{L^2(\Omega)}\leq e^{-\lambda}\|\phi_{\eb,x_0}(-\Dx+1)^\frac{s}{2}u\|^2_{L^2(\Omega)}, \quad s\in(0,1),\quad \lambda\geq 0.
\end{equation} 
Combining \eqref{3.11}, \eqref{3.13} we find
\begin{multline}
\frac{d}{d\lambda}\|v-\phi_{\eb,x_0}w\|^\frac{2}{1+s}_{L^2(\Omega)}+\frac{1}{1+s}\|v-\phi_{\eb,x_0}w\|^\frac{2}{1+s}_{L^2(\Omega)}\leq\\
 C_s\eb e^{-\frac{\lambda}{1+s}}\|\phi_{\eb,x_0}(-\Dx+1)^\frac{s}{2}u\|^\frac{2}{1+s}_{L^2(\Omega)},\quad \lambda\geq 0,\ s\in(0,1)
\end{multline}
with some constant $C_s$ depending on $s$. Multiplying the above inequality by $e^\frac{\lambda}{1+s}$ one finds
\begin{equation}
\frac{d}{d\lambda}\left(e^\frac{\lambda}{1+s}\|v-\phi_{\eb,x_0}w\|^\frac{2}{1+s}_{L^2(\Omega)}\right)\leq C_s\eb\|\phi_{\eb,x_0}(-\Dx+1)^\frac{s}{2}u\|^\frac{2}{1+s}_{L^2(\Omega)},\quad\lambda\geq 0,
\end{equation}
that is 
\begin{equation}
\|v(\lambda)-\phi_{\eb,x_0}w(\lambda)\|_{L^2(\Omega)}^\frac{2}{1+s}\leq C_s\eb \lambda e^{-\frac{\lambda}{1+s}}\|\phi_{\eb,x_0}(-\Dx+1)^\frac{s}{2}u\|^\frac{2}{1+s}_{L^2(\Omega)},\quad \lambda\geq 0,
\end{equation}
and hence
\begin{equation}
\|v(\lambda)-\phi_{\eb,x_0}w(\lambda)\|_{L^2(\Omega)}\leq C_s\eb^\frac{1+s}{2}\lambda^\frac{1+s}{2}e^{-\frac{\lambda}{2}}\|\phi_{\eb,x_0}(-\Dx+1)^\frac{s}{2}u\|_{L^2(\Omega)},\quad \lambda\geq 0,
\end{equation}
as long as $\eb\in(0,\eb_0]$, where constant $C_s> 0$ depends only on $s$. Now let us estimate the commutator
\begin{multline}
\|\phi_{\eb,x_0}(-\Dx+1)^\theta u-(-\Dx+1)^\theta(\phi_{\eb,x_0}u)\|_{L^2(\Omega)}\leq\\
C_s\eb^\frac{1+s}{2}\frac{1}{|\Gamma(-\theta)|}\int_0^\infty\lambda^{\frac{1+s}{2}-\theta-1}e^{-\frac{\lambda}{2}}d\lambda\|\phi_{\eb,x_0}(-\Dx+1)^\frac{s}{2}u\|_{L^2(\Omega)}=\left|r=\frac{\lambda}{2}\right|=\\
C_s\eb^\frac{1+s}{2}2^{\frac{1+s}{2}-\theta}\frac{1}{|\Gamma(-\theta)|}\int_0^\infty r^{\left(\frac{1+s}{2}-\theta\right)-1}e^{-r}dr\|\phi_{\eb,x_0}(-\Dx+1)^\frac{s}{2}u\|_{L^2(\Omega)}=\\
C_s\eb^\frac{1+s}{2}2^{\frac{1+s}{2}-\theta}\frac{\Gamma\left(\frac{1+s}{2}-\theta\right)}{|\Gamma(-\theta)|}\|\phi_{\eb,x_0}(-\Dx+1)^\frac{s}{2}u\|_{L^2(\Omega)},
\end{multline} that completes the proof.
\end{proof}
\begin{rem} Obviously, the above proposition allows to extend Corollary \ref{cor equiv n1} to the case when $\theta\in (0,1)$, that we will use extensively.
\end{rem} 

%%%%%%%%%%%%%%%%%%%%%%%%%%%%%%%%%%%%%%%%%%%%%%%%%%%%%
In the sequel we also need a commutator estimate involving compactly supported weight function $\psi_{x_0}$ such that
\begin{equation}
\label{psi}
\psi_{x_0}\in C^\infty_{0}(\R^3),\ {\rm supp}\psi_{x_0}\subset B^2_{x_0},\ \psi_{x_0}(x)\equiv 1\ {\rm for}\ x\in B^1_{x_0},
\end{equation}
where $x_0$ is some point of $\R^3$
\begin{prop}
\label{prop comest3}
Let $\theta\in(0,\frac{1}{2})$, $\eb\in(0,\eb_0]$ with small enough $\eb_0$ and $u\in L^2_{\{\eb,x_0\}}(\Omega)$. Then the following commutator estimate  holds true
\begin{equation}
\|(-\Dx+1)^\theta(\psi_{x_0} u)-\psi_{x_0}(-\Dx+1)^\theta u\|_{L^2(\Omega)}\leq C2^{\frac{1}{2}-\theta}\frac{\Gamma\left(\frac{1}{2}-\theta\right)}{|\Gamma(-\theta)|}\|\phi_{\eb,x_0}u\|_{L^2(\Omega)},
\end{equation}
for some absolute constant $C$ and small enough $\eb\in(0,\eb_0]$, where  $\psi_{x_0}$ is defined by \eqref{psi}.
\end{prop}
\begin{proof}
Similar to Proposition \ref{prop comest1} we introduce $v(\lambda):=e^{-(-\Dx +1)\lambda}(\psi_{x_0}u)$ and $w(\lambda):=e^{-(-\Dx +1)\lambda}u$. Repeating \eqref{3.1}-\eqref{3.3}
with $\psi_{x_0}$ instead of $\phi_{\eb,x_0}$ we deduce analogue of \eqref{3.3} 
\begin{multline}
\label{3.3 psi}
\frac{1}{2}\frac{d}{d\lambda}\|v-\phi_{\eb,x_0}w\|^2_{L^2(\Omega)}+\|\nabla(v-\phi_{\eb,x_0}w)\|^2_{L^2(\Omega)}+\|v-\phi_{\eb,x_0}w\|^2_{L^2(\Omega)}=\\
2(w\nabla\psi_{x_0}\nabla(v-\psi_{x_0}w))+(w\Dx\psi_{x_0},v-\psi_{x_0}w):=A_1+A_2.
\end{multline}
Due to the fact the $\psi_{x_0}$ is compactly supported the $A_1$ term admits the estimate
\begin{multline}
\label{3.4 psi}
|A_1|\leq C \|w,\Omega\cap B^2_{x_0}\|_{0,2}\|\nabla(v-\psi_{x_0}w)\|_{L^2(\Omega)}\leq C\|\phi_{\eb,x_0}w\|_{L^2(\Omega)}\|\nabla(v-\psi_{x_0}w)\|^2_{L^2(\Omega)}\leq\\
C\|\phi_{\eb,x_0}w\|^2_{L^2(\Omega)}+\frac{1}{2}\|\nabla(v-\psi_{x_0}w)\|_{L^2(\Omega)},
\end{multline}
for some absolute constant $C$, assuming that $\eb\in(0,\eb_0]$.

The $A_2$ term can be estimated analogously
\begin{equation}
\label{3.5 psi}
|A_2|\leq C\|\phi_{\eb,x_0}w\|^2_{L^2(\Omega)}+\frac{1}{2}\|v-\psi_{x_0}w\|^2_{L^2(\Omega)},
\end{equation}
for some absolute constant $C$, assuming that $\eb\in(0,\eb_0]$.

Combining \eqref{3.3 psi}-\eqref{3.5 psi} together with \eqref{3.6} we derive
\begin{equation}
\frac{d}{d\lambda}\|v-\psi_{x_0}w\|^2_{L^2(\Omega)}+\|v-\psi_{x_0}w\|^2_{L^2(\Omega)}\leq
 Ce^{-\lambda}\|\phi_{\eb,x_0}u\|^2_{L^2(\Omega)},
\end{equation}
for some absolute constant $C$, assuming that $\eb\in(0,\eb_0]$ and $\eb_0$ is small enough. Arguing along the lines of \eqref{3.7} -\eqref{3.7.1} we complete the proof. 
\end{proof}

Let $\phi,u$ be two functions on $\Omega$. Throughout the work we use the following notations for the commutator
\begin{align}
&[\phi,(-\Dx+1)^s]u:=\phi(-\Dx+1)^su-(-\Dx+1)^s(\phi u),\ s\in(0,1),\\
&[(-\Dx+1)^s,\phi]u:=(-\Dx+1)^s(\phi u)-\phi(-\Dx+1)^su,\ s\in(0,1).
\end{align}

\section{Existence of infinite-energy solutions}
\label{s.ex}
The aim of this section is to establish existence of infinite-energy solutions to the following semi-linear damped wave equation
\begin{equation}
\begin{cases}
\label{eq main}
\Dt^2u+\gamma(-\Dx+1)^\frac{1}{2}\Dt u-\Dx u+\lambda_0 u +f(u)=g(x),\quad x\in\Omega,\\
u|_{\d\Omega}=0,\ \xi_u(t)|_{t=0}=(u,\d_t u)|_{t=0}=(u_0,u_1),
\end{cases}
\end{equation}
where constants $\gamma$ and $\lambda_0$ are strictly positive and non-linearity $f\in C^1(\R)$ of critical growth and subject to assumptions \eqref{f.growth}, \eqref{f(s)s>-M}.

Further it is convenient to use the notations
\begin{align}
\label{E phi}
&\Ee=\Cal H^1_{\{\eb,x_0\}}(\Omega)\times L^2_{\{\eb,x_0\}}(\Omega),\\
\label{E b}
&\Cal{E}_b=\(H^1_b(\Omega)\cap\{u|_{\d\Omega}=0\}\)\times L^2_{b}(\Omega).
\end{align}
The norms in the corresponding spaces are given by
\begin{align}
\label{E phi norm}
& \xi=(\xi_1,\xi_2)\in \Ee,\ \|\xi\|^2_{\Ee}=\|\phi_{\eb,x_0}\nabla \xi_1\|^2_{L^2(\Omega)}+\lambda_0\|\phi_{\eb,x_0}\xi_1\|^2_{L^2(\Omega)}+\|\phi_{\eb,x_0}\xi_2\|^2_{L^2(\Omega)},\\
\label{E b norm}
& \xi=(\xi_1,\xi_2)\in \E_b,\ \|\xi\|^2_{\E_b}=\|\nabla \xi_1\|^2_{L^2_b(\Omega)}+\lambda_0\|\xi_1\|^2_{L^2_b(\Omega)}+\|\xi_2\|^2_{L^2_b(\Omega)}.
\end{align}
Now let us give the definition of infinite-energy solutions
\begin{Def}
\label{def inf.sol}
Function $u(t)$ such that $\xi_u(t)\in L^\infty([0,T];\E_b)$ and 
\begin{equation}
\label{4.0}
\sup_{x_0\in\R^3}\int_0^T\|\phi_{\eb_0,x_0}(-\Dx+1)^{\frac{1}{2}}\Dt u(t)\|^2_{L^2(\Omega)}dt<\infty,
\end{equation}
for some $\eb_0>0$, to be called infinite energy solution of problem \eqref{eq main} with initial data $(u_0,u_1)\in\E_b$ iff
\begin{multline}
-\int_0^T(\Dt u,\Dt \phi)dt+\gamma\int_0^T(\Dt u,(-\Dx+1)^\frac{1}{2} \phi)dt+\int_0^T(\nabla u,\nabla \phi)dt+\lambda_0\int_0^T(u,\phi)dt\\
\int_0^T(f(u),\phi)dt=\int_0^T(g,\phi)dt,\ \forall\phi\in C^\infty_0((0,T)\times\Omega),
\end{multline}
and $\xi_u|_{t=0}=(u_0,u_1)$.
\begin{rem}\label{rem in data 0}
The definition of solution does not depend on $\eb_0$ in \eqref{4.0}. Indeed, if $\eb\geq\eb_0$ it is straightforward. If $0<\eb<\eb_0$ the statement follows due to the fact that inequality \eqref{4.0} is \emph{uniform} with respect to $x_0$.
\end{rem}

\begin{rem} \label{rem in data}Initial data can be understood as follows. Definition \ref{def inf.sol} implies that both $u$ and $\Dt u$ belong to $L^\infty([0,T];L^2_{\{\eb,x_0\}}(\Omega))$ for arbitrary $\eb>0$, and hence $u\in C([0,T];L^2_{\{\eb,x_0\}}(\Omega))$ for any $\eb>0$. Also writing down the formula for $\Dt^2 u$ from \eqref{eq main}, taking into account Definition \ref{def inf.sol}, we see that $\Dt^2 u\in L^\infty([0,T];\Cal H^{-1}_{\{\eb,x_0\}}(\Omega))$ for any $\eb>0$. This together with $\Dt u\in L^\infty([0,T];\Cal H^{-1}_{\{\eb,x_0\}}(\Omega))$ implies
that $\Dt u\in C([0,T];\Cal H^{-1}_{\{\eb,x_0\}}(\Omega))$. Thus initial data can be naturally understood if we consider $\xi_u$ as continuous function in $\E^{-1}_{\{\eb,x_0\}}:=L^2_{\{\eb,x_0\}}(\Omega)\times \Cal H^{-1}_{\{\eb,x_0\}}(\Omega)$. Furthermore, by classical result (see Lemma 8.1, Chapter 3, \cite{LM}) we conclude that $\xi_u(t)$ is weakly continuous in $\Ee$, that is $\xi_u(t)\in C^w([0,T];\Ee)$. 
\end{rem}  

\end{Def} 
Now we are ready to state the result on the existence of infinite-energy solutions 
\begin{theorem}
\label{th ex}
Let non-linearity $f$ satisfy assumptions \eqref{f.growth}, \eqref{f(s)s>-M}, $\Omega$ be a smooth unbounded domain satisfying \eqref{2.6}, \eqref{2.7}, constants  $\gamma,\lambda_0>0$ and $g\in L^2_b(\Omega)$. Then problem \eqref{eq main} possesses at least one infinite-energy solution $u$ for any initial data $\xi_0=(u_0,u_1)\in\E_b$ on arbitrary segment $[0,T]$. Furthermore, there exists $\eb_0>0$ small enough, such that for every $\eb\in(0,\eb_0]$ the following estimate holds
\begin{multline}
\label{est en}
\|\xi_u(t)\|^2_{\E_b}+\sup_{x_0\in\Omega}\int_{\max\{0,t-1\}}^t\|\phi_{\eb,x_0}(-\Dx+1)^\frac{1}{4}\Dt u(s)\|^2_{L^2(\Omega)}ds\leq\\
 Q_{\eb}(\|\xi_0\|_{\E_b})e^{-\beta t}+Q_{\eb}(\|g\|_{L^2_b(\Omega)}),\quad t\geq 0,
\end{multline} 
for some constant $\beta>0$ and monotone increasing function $Q_{\eb}$ which are independent of $u$ and $t$.
\end{theorem}
Since the existence of infinite-energy solutions is not the main subject of the work we give its proof in Appendix \ref{sec.Appendix}. 
\section{Extra regularity: finite-energy case}
\label{s.exreg.fe}
Before considering the case of infinite-energy solutions we would like to demonstrate the main idea of the work in the simpler case of Dirichlet problem for equation \eqref{eq main} in a smooth \emph{bounded} domain $\Omega$. The energy space $\E$ associated with problem \eqref{eq main} in this case is given by
\begin{align}
\label{E}
&\E=H^1_0(\Omega)\times L^2(\Omega),\\
\label{E norm}
&\xi=(\xi_1,\xi_2)\in\E,\ \|\xi\|^2_{\E}=\|\nabla\xi_1\|^2_{L^2(\Omega)}+\lambda_0\|\xi_1\|^2_{L^2(\Omega)}+\|\xi_2\|^2_{L^2(\Omega)}.
\end{align}
The energy solutions, as previously, are understood in the sense of distributions, one just should substitute $\phi_{\eb,x_0}$ by $1$ in Definition \ref{def inf.sol}. 

Let us recall the result on the existence of finite-energy solutions to problem \eqref{eq main} in smooth bounded domains (see \cite{KZ2009,SZ2014}). This result can be also obtained as in Theorem \ref{th ex} with even simpler proof, since in this case we do not need to work with weighted spaces and use commutator estimates.
\begin{theorem}
\label{th ex.bdd}
Let non-linearity $f$ satisfy assumptions \eqref{f.growth}, \eqref{f(s)s>-M}, $\Omega$ be a smooth \emph{bounded} domain satisfying \eqref{2.6}, \eqref{2.7}, constants  $\gamma,\lambda_0>0$ and $g\in L^2(\Omega)$. Then problem \eqref{eq main} possesses at least one finite-energy (energy for brevity) solution $u$ for any initial data $\xi_0=(u_0,u_1)\in\E$ on arbitrary segment $[0,T]$. Furthermore,  the following estimate holds
\begin{equation}
\label{est en.bdd}
\|\xi_u(t)\|^2_{\E}+\int_{\max\{0,t-1\}}^t\|(-\Dx+1)^\frac{1}{4}\Dt u(s)\|^2_{L^2(\Omega)}ds\leq\\
 Q(\|\xi_0\|_{\E})e^{-\beta t}+Q(\|g\|_{L^2(\Omega)}),\quad t\geq 0,
\end{equation} 
for some constant $\beta>0$ and monotone increasing function $Q$ which are independent of $u$ and $t$. 
\end{theorem}  
\begin{rem}
\label{rem 5.1}
Theorem \ref{th ex.bdd} is valid for arbitrary $q>4$ as well if one substitutes energy space $\E$ by $\E_q=H^1_0(\Omega)\cap L^{q+2}(\Omega)\times L^2(\Omega)$ (see \cite{CV} for the case $\theta=0$, for $\theta\in(0,1]$ it is analogous).
\end{rem} 

As it was noticed in \cite{SZ2014} the only energy estimate \eqref{est en.bdd} is not enough to obtain \emph{uniqueness} of the energy solutions to the considered problem in the quintic case and one needs more information on the regularity of solutions to prove it. In the next theorem we would like to show how the required extra regularity can be obtained in a more optimal way in comparison with \cite{SZ2014}. 

\begin{theorem}
\label{th adreg1.bdd}
Let assumption of Theorem \ref{th ex.bdd} be satisfied. Then for every initial data $\xi_0\in\E$ the energy solution of problem \eqref{eq main} possesses extra regularity $u\in L^4([0,T];L^{12}(\Omega))$ on arbitrary segment $[0,T]$ and the following estimate holds
\begin{equation}
\label{est adreg1.bdd}
\int_{\max\{0,t-1\}}^t\|u(s)\|^4_{L^{12}(\Omega)}ds\leq Q(\|\xi_0\|_{\E})e^{-\beta t}+Q(\|g\|_{L^2(\Omega)}), \quad t\geq 0,
\end{equation}
for some constant $\beta>0$ and monotone increasing function $Q$ which are independent of $u$ ant $t$.
\end{theorem}
\begin{proof}
At the moment we would like to explain the idea. The rigorous proof in the more difficult case of infinite-energy solutions is given in the next section.

Let us, formally, multiply equation \eqref{eq main} by $u^3$ (exactly this step should be justified). One finds
\begin{multline}
\frac{d}{dt}\left((\Dt u,u^3)+\gamma((-\Dx+1)^\frac{1}{2}u,u^3)\right)+3(|\nabla u|^2,u^2)+\lambda_0\|u\|^4_{L^4(\Omega)}+(f(u)u,u^2)=\\
(g,u^3)+3\gamma((-\Dx+1)^\frac{1}{2}u,u^2\Dt u)+3(u^2,(\Dt u)^2).
\end{multline}
The desired extra regularity comes from the estimate
\begin{equation}
\|u(t)\|^4_{L^{12}(\Omega)}=\|u^2(t)\|^2_{L^6(\Omega)}\leq C\|\nabla(u^2(t))\|^2_{L^2(\Omega)}\leq C\int_\Omega |\nabla u(t)|^2|u(t)|^2dx.
\end{equation} 
It appears that all other terms are subordinated. In particular, thanks to the extra smoothness of $\Dt u$ induced by the fractional damping and continuous embedding $\Cal H^\frac{1}{2}_{\{0,0\}}(\Omega)\subset L^3(\Omega)$ the bad term coming from $\Dt^2 u$ can be controlled as follows
\begin{equation}
(u^2,(\Dt u)^2)\leq \|u\|^2_{L^6(\Omega)}\|\Dt u\|^2_{L^3(\Omega)}\leq\|\nabla u\|^2\|(-\Dx+1)^\frac{1}{4}\Dt u\|^2_{L^2(\Omega)}.
\end{equation}
Also due to Holder inequality with exponents $2$, $3$ and $6$  we have
\begin{multline}
|((-\Dx+1)^\frac{1}{2}u,\Dt u\cdot u^2)|\leq C\|\nabla u\|_{L^2(\Omega)}\|\Dt u\|_{L^3(\Omega)}\|u\|^2_{L^{12}(\Omega)}\leq\\
 \eb \|u\|^4_{L^{12}(\Omega)}+C_\eb \|(-\Dx+1)^\frac{1}{4}\Dt u\|^2_{L^2(\Omega)}\|\nabla u\|^2_{L^2(\Omega)}, 
\end{multline}
for arbitrary $\eb>0$ which will be fixed small below.

Finally due to dissipative assumption \eqref{f(s)s>-M} we have
\begin{equation}
(f(u)u,u^2)\geq -M\|u\|^2_{L^2(\Omega)}.
\end{equation}
Combining the above estimates, and fixing $\eb>0$ small enough we derive
\begin{multline}
\frac{d}{dt}(\Dt u,u^3)+\kappa\|u\|^4_{L^{12}(\Omega)}\leq\\
 C\|(-\Dx+1)^\frac{1}{4}\Dt u\|^2_{L^2(\Omega)}\|\nabla u\|^2_{L^2(\Omega)}+\|g\|_{L^2(\Omega)}\|\nabla u\|^3_{L^2(\Omega)}+M\|u\|^2_{L^2(\Omega)},
\end{multline}
for some $\kappa>0$.

Integrating the above inequality from $\max\{0,t-1\}$ to $t$ and using dissipative estimate \eqref{est en.bdd} one derives the required result.
\end{proof}
\begin{rem}
\label{rem 5.2}
Due to Remark \ref{rem 5.1} and the fact that the above proof does not essentially use growth assumption \eqref{f.growth} we conclude that Theorem \ref{th adreg1.bdd} remains valid for arbitrary $q>4$ either if the corresponding energy space $\E$ is changed to $\E_q$. However, it seems that for the case $q>4$ the obtained extra regularity is still insufficient for the uniqueness.    
\end{rem}
\begin{rem}
\label{rem 5.3}
In contrast to the case of pure wave equation (\cite{Sogge2009,stri}) or damped wave equation (\cite{KSZ,S ADE}), the uniqueness of energy solutions for problem \eqref{eq main} which is obtained based on this regularity is unconditional, that is valid for \emph{any} energy solution. 
\end{rem}
\section{Extra regularity: infinite-energy case}
\label{s.exreg.ie}
Now we would like to focus on additional regularity of infinite-energy solutions.
\begin{theorem}
\label{th adreg1}
Let assumptions of Theorem \ref{th ex} be satisfied and $\eb\in(0,\eb_0]$ be small enough. Then for every initial data $\xi_0=(u_0,u_1)\in\E_b$ the infinite-energy solution of problem \eqref{eq main} possesses extra regularity $u\in L^4([0,T];L^{12}_{\{\eb,x_0\}}(\Omega))$ on arbitrary segment $[0,T]$ and the following estimate holds
\begin{equation}
\label{est adreg1}
\sup_{x_0\in \Omega}\int_{\max\{0,t-1\}}^t\|\phi_{\eb,x_0}u(s)\|^4_{L^{12}(\Omega)}ds\leq Q_\eb(\|\xi_0\|_{\E_b})e^{-\beta t}+Q_\eb(\|g\|_{L^2_b(\Omega)}), \quad t\geq 0,
\end{equation}
for some constant $\beta>0$ and monotone increasing function $Q_\eb$ which are independent of $u$ ant $t$.
\end{theorem} 
\begin{proof}

We would like to adapt and use the idea presented in Theorem \ref{th adreg1.bdd} in our infinite-energy case as well as make it rigorous. To this end we would like to multiply equation \eqref{eq main} by $\phi^4_{\eb,x_0}(x)u^3(x)$. The main technical problem in justifying this multiplication is the fact that $\int_\Omega f(u(t))u^3(t)\phi^4_{\eb,x_0}(x)dx$ may be infinite since at the moment we know that $u\in L^\infty([0,T];L^6_b(\Omega))$ only and $f(u)u^3$ is of order $u^8$ due to $\eqref{f.growth}$. Therefore to keep things rigorous we will approximate $r^3$ by $C^1(\R)$ function $\psi_n(r)$ given by the formula
\begin{equation}
\label{4.1}
\psi_n(r)=
\begin{cases}
r^3,\quad |r|\leq n,\\
3n^2r-2n^3,\quad r\geq n,\\
3n^2r+2n^3,\quad r\leq -n.
\end{cases}
\end{equation}  
Obviously, $\psi_n(r)$ is such that $|\psi_n(r)|\leq C_n |r|$ when $C_n$ is large enough and consequently the product $(f(u),\phi^4_{\eb,x_0}(x)\psi_n(u))$ makes sense. We also notice that $\psi_n(u)\in H^1_b(\Omega)$, this implies that product  $(-\Dx u,\phi^4_{\eb,x_0}\psi_n(u))$ is finite. And from the formula for $\Dt^2 u$ given by the equation \eqref{eq main} we conclude that product $(\Dt^2 u,\phi^4_{\eb,x_0}u)$ is also finite and thus desired multiplication makes sense and leads to 
\begin{multline}
\label{4.2}
\frac{d}{dt}\left((\phi_{\eb,x_0}\Dt u,\phi^3_{\eb,x_0}\psi_n(u))+\gamma(\phi_{\eb,x_0}(-\Dx+1)^\frac{1}{2}u,\phi^3_{\eb,x_0}\psi_n(u))\right)+\\
(\phi^2_{\eb,x_0}|\nabla u|^2,\phi^2_{\eb,x_0}\psi_n'(u))+\frac{\lambda_0}{2}(u,\psi_n(u)\phi^4_{\eb,x_0})+(\frac{\lambda_0}{2}u+f(u),\psi_n(u)\phi^4_{\eb,x_0})=\\
(\phi^2_{\eb,x_0}|\Dt u|^2,\phi^2_{\eb,x_0}\psi_n'(u))-4(\nabla u, \phi^3_{\eb,x_0}\nabla \phi_{\eb,x_0}\psi_n(u))+\\
\gamma (\phi_{\eb,x_0}(-\Dx+1)^\frac{1}{2}u,\phi^2_{\eb.x_0}\psi_n'(u)\phi_{\eb,x_0}\Dt u)+(g,\phi^4_{\eb,x_0}\psi_n(u)):=
A_1+A_2+A_3+A_4,
\end{multline}
where $(\cdot,\cdot)$ denotes scalar product in $L^2(\Omega)$.

Let us estimate each $A_i$ separately. $A_1$-term can be estimated just using Holder inequality with exponents $\frac{3}{2}$ and $3$
\begin{multline}
\label{A1}
A_1\leq \|\Dt u\|^2_{\{\eb,x_0\},0,3}\(\int_\Omega\phi^6_{{\eb},x_0}(\psi_n'(u))^3dx\)^\frac{1}{3}=\\
\|\Dt u\|^2_{\{\eb,x_0\},0,3}\(27\int_{\{x:|u(t)|\leq n\}}\phi^6_{{\eb},x_0}u^6dx+27\int_{\{x: |u(t)|\geq n\}}\phi^6_{\eb,x_0}n^6dx\)^\frac{1}{3}\leq
3\|\Dt u\|^2_{\{\eb,x_0\},0,3}\|u\|^2_{\{\eb,x_0\},0,6}.
\end{multline}

$A_2$-term can be estimated just by Cauchy inequality and the fact that $|\nabla \phi_{\eb,x_0}|\leq \eb \phi_{\eb,x_0}$
\begin{equation}
\label{A2}
|A_2|\leq 4\eb\|\nabla u\|_{\{\eb,x_0\},0,2}\|\phi^3_{\eb,x_0}\psi_n(u)\|_{L^2(\Omega)}\leq 16\eb \|\nabla u\|_{\{\eb,x_0\},0,2}\|u\|^3_{\{\eb,x_0\},0,6}.
\end{equation}

With $A_3$-term we should be more careful. Using Holder inequality with exponents $2$, $6$, and $3$, continuous embedding $\Cal H^\frac{1}{2}_{\{\eb,x_0\}}(\Omega)\subset L^3_{\{\eb,x_0\}}(\Omega)$ we derive 
\begin{multline}
\label{4.3}
A_3\leq \gamma\|u\|_{\{\eb,x_0\},1,2}\|\phi^2_{\eb,x_0}\psi'_n(u)\|_{\{\eb,x_0\},0,6}\|\phi_{\eb,x_0}(-\Dx+1)^\frac{1}{4}\Dt u\|_{L^2(\Omega)}\leq\\
 k\|\phi^2_{\eb,x_0}\psi'_n(u)\|^2_{\{\eb,x_0\},0,6}+C_k\|\phi_{\eb,x_0}(-\Dx+1)^\frac{1}{4}\Dt u\|^2_{L^2(\Omega)}\|u\|^2_{\{\eb,x_0\},1,2},
\end{multline} 
where $k$ will be fixed small enough below.

Further, it is convenient to use auxiliary function $\Psi_n(r)$ such that 
\begin{align}
\label{4.4}
\psi_n'(r)=|\Psi_n'(r)|^2,\\
\label{4.5}
|\psi_n'(r)|\leq C\Psi_n(r),
\end{align} 
for some absolute constant $C$. It is easy to see that $\Psi_n(r)$ can be chosen as follows
\begin{equation}
\label{Psi}
\Psi_n(r)=\begin{cases}
\frac{\sqrt{3}}{2}r^2,\quad &|r|\leq n,\\
\sqrt{3}n|r|-\frac{\sqrt{3}}{2}n^2,\quad &|r|\geq n.
\end{cases}
\end{equation}
Then from \eqref{4.3} and \eqref{4.5} we find
\begin{equation}
\label{A3}
A_3\leq kC\|\phi^2_{\eb,x_0}\Psi_n(u)\|^2_{\{\eb,x_0\},0,6}+C_k\|\phi_{\eb,x_0}(-\Dx+1)^\frac{1}{4}\Dt u\|^2_{L^2(\Omega)}\|u\|^2_{\{\eb,x_0\},1,2},
\end{equation}
where $k$ is small enough to be chosen below, and is an absolute constant.

On the other hand, due to \eqref{4.4} and continuous embedding $H^1_0(\Omega)\subset L^6(\Omega)$ we have
\begin{multline}
\|\phi^2_{\eb,x_0}\Psi_n(u)\|^2_{L^6(\Omega)}\leq C\(\|\phi^2_{\eb,x_0}\Psi_n(u)\|^2_{L^2(\Omega)}+\|\nabla(\phi^2_{\eb,x_0}\Psi_n(u))\|^2_{L^2(\Omega)}\)\leq\\
C\(\|\phi^2_{\eb,x_0}\Psi_n(u)\|^2_{L^2(\Omega)}+(\phi^4_{\eb,x_0}|\nabla u|^2, \psi_n'(u))\).
\end{multline} 
And using the fact that $\Psi_n^2(r)\leq \psi_n(r)r$ we proceed the previous inequality as follows
\begin{equation}
\|\phi^2_{\eb,x_0}\Psi_n(u)\|^2_{L^6(\Omega)}\leq C\(\frac{\lambda_0}{2}(\phi^4_{\eb,x_0}\psi_n(u),u)+(\phi^4_{\eb,x_0}|\nabla u|^2, \psi_n'(u))\),
\end{equation}
that is
\begin{equation}
\label{4.6}
\frac{\lambda_0}{2}(\phi^4_{\eb,x_0}\psi_n(u),u)+(\phi^4_{\eb,x_0}|\nabla u|^2, \psi_n'(u))\geq \mu \|\phi^2_{\eb,x_0}\Psi_n(u)\|^2_{L^6(\Omega)},
\end{equation}
for some constant $\mu >0$.

Estimate of $A_4$-term is trivial
\begin{equation}
\label{A4}
A_4\leq C\|g\|_{\{\eb,x_0\},0,2}\|u\|^3_{\{\eb,x_0\},0,6}\leq C(\|g\|^2_{\{\eb,x_0\},0,2}+\|u\|^6_{\{\eb,x_0\},0,6}),
\end{equation}
with some absolute constant $C>0$.

So it remains to estimate the term involving $f(u)$ on the left hand side of \eqref{3.2}. Noticing that $\(\frac{\lambda_0}{2}u+f(u)\)\psi_n(u)\geq 0$, when $|u|>n$ ($n$ is large enough) and using \eqref{f(s)s>-M} we deduce
\begin{equation}
\label{4.7}
\int_\Omega\phi^4_{\eb,x_0}\(\frac{\lambda_0}{2}u+f(u)\)\psi_n(u)dx\geq -M\|u\|^2_{\{\eb,x_0\},0,2}.
\end{equation}

Combining \eqref{4.2}, \eqref{A1}, \eqref{A2}, \eqref{A3} (with small $k$), \eqref{4.6}, \eqref{A4} we conclude
\begin{multline}\label{4.8}
\frac{d}{dt}\left((\phi_{\eb,x_0}\Dt u,\phi^3_{\eb,x_0}\psi_n(u))+\gamma(\phi_{\eb,x_0}(-\Dx+1)^\frac{1}{2}u,\phi^3_{\eb,x_0}\psi_n(u))\right)+\beta \|\phi^2_{\eb,x_0}\Psi_n(u)\|^2_{L^6(\Omega)} \leq\\
C\(\|u\|^2_{\{\eb,x_0\},1,2}+\|u\|^4_{\{\eb,x_0\},1,2}+\|u\|^6_{\{\eb,x_0\},1,2}+\|\phi_{\eb,x_0}(-\Dx+1)^\frac{1}{4}\Dt u\|^2_{L^2(\Omega)}\|u\|^2_{\{\eb,x_0\},1,2} +\|g\|^2_{\{\eb,x_0\},0,2}\).
\end{multline}

Eventually, integration of the above inequality from $\max\{0,t-1\}$ to $t$ and dissipative estimate \eqref{est en} together with Corollary \ref{cor wsp4} yields
\begin{equation}
\int_{max\{0,t-1\}}^t\|\phi^2_{\eb,x_0}\Psi_n(u(s))\|^2_{L^6(\Omega)}ds\leq Q_\eb(\|\xi_0\|_{\E_b})e^{-\beta t}+Q_\eb(\|g\|_{L^2_b(\Omega)}),\quad t\geq 0, 
\end{equation}
\emph{uniformly} with respect to $n$ and $x_0$ for some monotone increasing function $Q_\eb$, constant $\beta>0$. Passing to the limit as $n$ goes to infinity in the above estimate, for instance by Fatou's lemma, and taking supremum in $x_0$ we complete the proof.
\end{proof}

It will be shown in Theorem \ref{th adreg2} that with estimate \eqref{est adreg1} in hands it is not difficult to check that actually $u\in L^2([0,T];\Cal H^\frac{3}{2}_{\{\eb,x_0\}}(\Omega))$.  
\section{Uniqueness of infinite-energy solutions}
\label{s.u}

We are ready to prove uniqueness of infinite-energy solutions. We emphasize that uniqueness, at least in quintic case, is essentially based on additional regularity of energy solutions obtained in Theorem \ref{th adreg1}. Also to prove the uniqueness, in addition to \eqref{E phi}-\eqref{E b}, we need to work in the energy space $\Elr$ with compactly supported weight $\psi_{x_0}$ (see \eqref{psi})
\begin{align}
\label{E psi}
& \Elr =\(H^1_{\psi_{x_0}}(\Omega)\cap\{u|_{\partial\Omega}=0\}\)\times L^2_{\psi_{x_0}}(\Omega),\\
& \xi=(\xi_1,\xi_2)\in \Elr,\ \|\xi\|^2_{\Elr}=\|\psi_{x_0}\nabla \xi_1\|^2_{L^2(\Omega)}+\lambda_0\|\psi_{x_0}\xi_1\|^2_{L^2(\Omega)}+\|\psi_{x_0}\xi_2\|^2_{L^2(\Omega)}.
\end{align}
\begin{theorem}
\label{th uniq cont}
Let assumptions of Theorem \ref{th ex} be satisfied, and $u_1$ and $u_2$ be two infinite-energy solutions of problem \eqref{eq main} with initial data $\xi^1_0,\ \xi^2_0\in\E_b$ respectively. Then, for arbitrary $T>0$, the following estimate holds
\begin{equation}
\label{cont.dep.id}
\|\xi_{u_1}(t)-\xi_{u_2}(t)\|^2_{\E_b}\leq Q(A_T)\|\xi^1_0-\xi^2_0\|^2_{\E_b}e^{Q(A_T)t},\ t\in[0,T],
\end{equation}
where
\begin{equation}\label{AT}
A_T=\sup_{x_0\in\Omega}exp\left\{C\int_0^T \(1+\|u_1(\tau)\|^4_{L^{12}(\Omega\cap B^2_{x_0})}+\|u_2(\tau)\|^4_{L^{12}(\Omega\cap B^2_{x_0})}\)d\tau\right\},
\end{equation}
for some monotone increasing function $Q$ and absolute constant $C$ which do not depend on $u_1$, $u_2$ and $t$. In particular, infinite-energy solution is unique and depends continuously on initial data.   
\end{theorem}
\begin{proof}
As usual let us consider $v=u_1-u_2$ which solves the problem
\begin{equation}
\label{eq u1-u2}
\begin{cases}
\Dt^2 v+\gamma(-\Dx+1)^\frac{1}{2}\Dt v-\Dx v+\lambda_0 v = f(u_2)-f(u_1),\quad x\in\Omega,\\
v|_{\d\Omega}=0,\  \xi_v|_{t=0}=\xi^1_0-\xi^2_0.
\end{cases}
\end{equation}
To get the desired estimate we need to multiply equation \eqref{eq u1-u2} by $\psi^2_{x_0}\Dt v$. We notice that now due to extra regularity given by Theorem \ref{th adreg1} and interpolation $[L^\infty([0,T];H^1_{\{\eb,x_0\}}(\Omega)), L^4([0,T];L^{12}_{\{\eb,x_0\}}(\Omega)]_\frac{4}{5}=L^5([0,T];L^{10}_{\{\eb,x_0\}}(\Omega))$ we know that $u_i\in L^5([0,T];L^{10}_{\{\eb,x_0\}}(\Omega))$ and hence $f(u_i)\in L^1([0,T];L^2_{\{\eb,x_0\}}(\Omega))$. And since $\Dt u\in L^2([0,T];H^\frac{1}{2}_{\{\eb,x_0\}}(\Omega))$ we have $(-\Dx+1)^\frac{1}{2}\Dt u\in L^2([0,T];H^{-\frac{1}{2}}_{\{\eb,x_0\}}(\Omega)) $ and thus the multiplication of equation \eqref{eq u1-u2} by $\psi^2_{x_0}\Dt v$ makes sense and yields 
\begin{multline}
\label{4.9}
\frac{1}{2}\frac{d}{dt}\|\xi_v\|^2_{\Elr}+\gamma((-\Dx+1)^\frac{1}{2}\Dt v,\psi^2_{x_0}\Dt v)=\\
-2(\nabla v,\psi_{x_0}\nabla\psi_{x_0}\Dt v)+(f(u_2)-f(u_1),\psi^2_{x_0}\Dt v).
\end{multline}
The term involving non-linearity on the right hand side of \eqref{4.9} can be estimated based on growth assumption on $f$ \eqref{f.growth}, Holder inequality and the fact that \emph{$\psi_{x_0}$ is compactly supported} (exactly at this place we have to use $\psi_{x_0}$ instead of $\phi_{\eb,x_0}$)
\begin{multline}
\label{est nl}
(f(u_2)-f(u_1),\psi^2_{x_0}\Dt v)\leq C\int_{\Omega}(1+|u_1|^4+|u_2|^4)\psi^2_{x_0}v\,\Dt v\,dx\leq\\
 C\(1+\|u_1\|^4_{L^{12}(\Omega\cap B^2_{x_0})}+\|u_2\|^4_{L^{12}(\Omega\cap B^2_{x_0})}\)\|\psi_{x_0}v\|_{L^6(\Omega)}\|\psi_{x_0}\Dt v\|_{L^2(\Omega)}\leq\\
 C\(1+\|u_1\|^4_{L^{12}(\Omega\cap B^2_{x_0})}+\|u_2\|^4_{L^{12}(\Omega\cap B^2_{x_0})}\)\(\|\nabla(\psi_{x_0}v)\|^2_{L^2(\Omega)}+\|\psi_{x_0}\Dt v\|^2_{L^2(\Omega)}\)\leq \\
 C\(1+\|u_1\|^4_{L^{12}(\Omega\cap B^2_{x_0})}+\|u_2\|^4_{L^{12}(\Omega\cap B^2_{x_0})}\)\(\|\nabla\psi_{x_0}v\|^2_{L^2(\Omega)}+\|\xi_v\|^2_{\Elr}\),
\end{multline} 
with some absolute constant $C$. To be able to use Gronwall's inequality further we absorb the extra term $\|\nabla\psi_{x_0}v\|^2$ by introducing modified energy
\begin{equation}
\label{Em}
E_{x_0}(\xi_v)=\|\nabla\psi_{x_0}v\|^2_{L^2(\Omega)}+\|\xi_v\|^2_{\Elr},
\end{equation}
and rewrite \eqref{4.9} in terms of $E_{x_0}(\xi_v)$
\begin{multline}
\label{4.9.01}
\frac{1}{2}\frac{d}{dt}E_{x_0}(\xi_v)+\gamma((-\Dx+1)^\frac{1}{2}\Dt v,\psi^2_{x_0}\Dt v)\leq\\
C\(1+\|u_1\|^4_{L^{12}(\Omega\cap B^2_{x_0})}+\|u_2\|^4_{L^{12}(\Omega\cap B^2_{x_0})}\)E_{x_0}(\xi_v)-2(\nabla v,\psi_{x_0}\nabla\psi_{x_0}\Dt v)+2(|\nabla\psi_{x_0}|^2v,\Dt v).
\end{multline}
The term with fractional laplacian can be rewritten as follows
\begin{multline}
\label{4.91}
((-\Dx+1)^\frac{1}{2}\Dt v,\psi^2_{x_0}\Dt v)=
\gamma(\left[\psi_{x_0},(-\Dx+1)^\frac{1}{4}\right](-\Dx+1)^\frac{1}{4}\Dt v,\psi_{x_0}\Dt v)\\
+\gamma(\psi_{x_0}(-\Dx+1)^\frac{1}{4}\Dt v,\left[(-\Dx+1)^\frac{1}{4},\psi_{x_0}\right]\Dt v)+
\|\psi_{x_0}(-\Dx+1)^\frac{1}{4}\Dt v\|^2_{L^2(\Omega)},
\end{multline}
that together with Proposition \ref{prop comest3} (for $\eb$ small enough) gives
\begin{multline}
\label{4.10}
\gamma((-\Dx+1)^\frac{1}{2}\Dt v,\psi^2_{x_0}\Dt v)\geq \gamma\|(-\Dx+1)^\frac{1}{4}\Dt v\|^2_{L^2(\Omega\cap B^1_{x_0})}\\-\frac{\delta}{2} \|\phi_{\eb,x_0}(-\Dx+1)^\frac{1}{4}\Dt v\|^2_{L^2(\Omega)}-C_\delta\|\Dt v(s)\|^2_{\{\eb,x_0\},0,2}.
\end{multline}
for arbitrary $\delta > 0$ which will be chosen small below.

From \eqref{4.9.01}, \eqref{4.10} and obvious estimates one derives 
\begin{multline}
\label{4.11}
\frac{d}{dt}E_{x_0}(\xi_v)+2\gamma\|(-\Dx+1)^\frac{1}{4}\Dt v\|^2_{L^2(\Omega\cap B^1_{x_0})}\leq\\
C\left(1+\|u_1\|^4_{L^{12}(\Omega\cap B^2_{x_0})}+\|u_2\|^4_{L^{12}(\Omega\cap B^2_{x_0})}\right)E_{x_0}(\xi_v)+\\
\delta \|\phi_{\eb,x_0}(-\Dx+1)^\frac{1}{4}\Dt v\|^2_{L^2(\Omega)}+C_\delta\|\xi_v\|^2_{\Ee}, {\rm\ for\ all\ }t\geq 0.
\end{multline}
Applying Gronwall's lemma to \eqref{4.11} we obtain
\begin{multline}\label{4.12}
E_{x_0}(\xi_v(t))+2\gamma\int_0^t\|(-\Dx+1)^\frac{1}{4}\Dt v(s)\|^2_{L^2(\Omega\cap B^1_{x_0})}ds\leq \\
A_T\(E_{x_0}(\xi_v(0))+\delta\int_0^t\|\phi_{\eb,x_0}(-\Dx+1)^\frac{1}{4}\Dt v(s)\|^2_{L^2(\Omega)}ds+C_\delta\int_0^t\|\xi_v(s)\|^2_{\Ee}ds\),\ t\in[0,T],
\end{multline}
where $A_T$ is given by \eqref{AT}.

Furthermore, multiplying \eqref{4.12} by $\phi^2_{\mu,x_0}$ with $0<\mu<\eb$, using Theorem \ref{th wsp2} on the left hand side of \eqref{4.12}, Theorem \ref{th wsp1} ($p=2$, $q=1$) on the right hand side of \eqref{4.12} and fixing $\delta =\delta(A_T)$ small enough we end up with
\begin{multline}\label{4.13}
\|\xi_v(t)\|^2_{\E_{\{\mu,x_0\}}}+\kappa\int_0^t\|\phi_{\mu,x_0}(-\Dx+1)^\frac{1}{4}\Dt v\|^2_{L^2(\Omega)}\leq\\ Q_{\eb,\mu}(A_T)\(\|\xi_v(0)\|^2_{\E_{\{\mu,x_0\}}}+\int_0^t\|\xi_v(s)\|^2_{\E_{\{\mu,x_0\}}}ds\),\ t\in[0,T],
\end{multline}
where constant $\kappa>0$ and $Q_{\eb,\mu}$ is some monotone increasing function. 
Applying Gronwall's lemma to \eqref{4.13} and taking supremum with respect to $x_0$ we finish the proof.
\end{proof}

\begin{rem}
We note that Remarks \ref{rem 5.1}-\ref{rem 5.3} hold true in the context of infinite-energy solutions if one substitutes $\E_b$ by $\E_{q,b}:=(H^1_b(\Omega)\cap L^{q+2}_b(\Omega)\cap\{u|_{\partial\Omega}=0\})\times L^2_b(\Omega)$.
\end{rem}

\section{Smoothing property of infinite-energy solutions}
\label{s.sm}
It is well known that the wave equation with fractional damping of the form $(-\Dx)^\theta\Dt u$ and $\theta\in(0,1)$ possesses smoothing property similar to parabolic equations (see \cite{KZ2009}, \cite{SZ2014}, \cite{S ADE} and references therein). Here we would like to establish smoothing property of problem \eqref{eq main} in the case of infinite-energy solutions. 

As usual, we start with the fact that infinite-energy solutions are smoother if initial data are smoother. Although the next Theorem gives the estimate which is divergent as $t\to+\infty$, this estimate will be improved at the end of the section.
\begin{theorem}
\label{th sm1}
Let assumptions of Theorem \ref{th ex} hold true and let initial data be such that
\begin{equation}
\xi_0\in \E^1_b := H^2_b(\Omega)\cap\{u|_{\d\Omega}=0\}\times H^1_b(\Omega)\cap\{u|_{\d\Omega}=0\}.
\end{equation}
Then $\xi_u(t)\in \E^1_b$ and $\xi_{\Dt u}(t)\in\E_b$ for all $t\geq 0$ and the following estimate holds
\begin{equation}
\label{est sm1}
\|\xi_u(t)\|^2_{\E^1_b}+\|\xi_{\Dt u}(t)\|^2_{\E_b}\leq Q_1(\|\xi_0\|_{\E^1_b},T)+Q_2(\|g\|_{b,0,2},T),\ t\in[0,T],
\end{equation}
for some monotone increasing functions $Q_1$ and $Q_2$ which are independent of $u$. 
\end{theorem}
\begin{proof}
Since the theorem is more or less standard we restrict ourselves to formal derivation of estimate \eqref{est sm1}. 

Let $v:=\Dt u$, then $v$ solves the problem
\begin{equation}
\label{eq ut}
\begin{cases}
\Dt^2v +\gamma(-\Dx+1)^\frac{1}{2}\Dt v-\Dx v+\lambda_0 v+f'(u)v=0,\\
v|_{\d\Omega}=0, \xi_v(0)=(\Dt u(0),\Dt^2 u(0))\in\E_b, 
\end{cases}
\end{equation}
where
\begin{equation}
\Dt^2 u(0):=-\gamma(-\Dx+1)^\frac{1}{2}u_1+\Dx u_0-\lambda_0u_0-f(u_0)+g.
\end{equation}
Let us check that $\xi_u(t)\in\E^1_b$ if and only if $\xi_v(t)\in\E_b$. Indeed,
from \eqref{eq main} we see, due to triangle inequality and continuous embedding $H^2_b(\Omega)\subset C_b(\Omega)$ (here $C_b(\Omega)$ denotes the space of continuous and bounded functions over $\Omega$), that 
\begin{equation}
\|\Dt^2 u(t)\|_{b,0,2}\leq Q(\|\xi_u(t)\|_{\E^1_b})+\|g\|_{b,0,2},
\end{equation}
and hence
\begin{equation}
\label{5.0}
\|\xi_v(t)\|_{\E_b}\leq Q(\|\xi_u(t)\|_{\E^1_b})+\|g\|_{b,0,2},
\end{equation}
for some monotone increasing $Q$.

And vice versa, rewriting equation \eqref{eq main} in the form
\begin{equation}
\label{5.1}
-\Dx u +\lambda_0 u+f(u)=g(x)-\Dt^2u-\gamma(-\Dx+1)^\frac{1}{2}\Dt u:=h(t)\in L^2_b(\Omega), 
\end{equation}
and multiplying equation \eqref{5.1} by $\phi^4_{\eb,x_0}u^3$, that can be justified as in Theorem \ref{th adreg1}, we deduce 
\begin{equation}
\|u(t)\|^4_{\{\eb,x_0\},0,12}\leq C\(1+\|u(t)\|^6_{\{\eb,x_0\},1,2}+\|h(t)\|^2_{\{\eb,x_0\},0,2}\)
\end{equation}  
uniformly with respect to $x_0$ and hence $u\in L^{12}_b(\Omega)$ and we have
\begin{equation}
\|u(t)\|^4_{b,0,12}\leq C\(1+\|u(t)\|^6_{b,1,2}+\|h(t)\|^2_{b,0,2}\).
\end{equation} 
Consequently $f(u)\in L^2_b(\Omega)$, indeed using \eqref{f.growth} we deduce
\begin{multline}
\label{5.2}
\|f(u)\|_{b,0,2}\leq C(1+\|u^5\|_{b,0,2})\leq  C(1+\|u\|^5_{b,0,10})\leq C(1+\|u\|^4_{b,0,12}\|u\|_{b,0,6})\leq\\
C(1+\|u\|^7_{b,1,2}+\|h\|^3_{b,0,2}),
\end{multline}
for some constant $C>0$. Combining \eqref{5.1}, \eqref{5.2} and using elliptic regularity we conclude that $u\in H^2_b(\Omega)$ and the following estimate holds
\begin{equation}
\|u\|^2_{b,2,2}\leq C(\|h\|^2_{b,0,2}+\|f(u)\|^2_{b,0,2})\leq Q(\|\xi_u\|_{\E_b})+Q(\|\xi_v\|_{\E_b})+C(1+\|g\|^6_{b,0,2})
\end{equation}
and hence
\begin{equation}
\label{5.3}
\|\xi_u\|^2_{\E^1_b}\leq Q(\|\xi_u\|_{\E_b})+Q(\|\xi_v\|_{\E_b})+C(1+\|g\|^6_{b,0,2}),
\end{equation}
for some monotone increasing function $Q$ and absolute constant $C$ which do not depend on $u$.

Thus to get estimate \eqref{est sm1} it suffices to get $\E_b$-norm control of $\xi_v(t)$. To this end we multiply equation \eqref{eq ut} by $\psi^2_{x_0}\Dt v$ 
\begin{equation}
\frac{1}{2}\frac{d}{dt}\|\xi_v(t)\|^2_{\Elr}+\gamma((-\Dx+1)^\frac{1}{2}\Dt v,\psi^2_{x_0}\Dt v)=\\ 
-(f'(u)v,\psi^2_{x_0}\Dt v)-2(\nabla v \nabla\psi_{x_0},\psi_{x_0}\Dt v).
\end{equation}
Arguing along the lines \eqref{est nl}-\eqref{4.11} we find
\begin{multline}
\label{5.3.1}
\frac{d}{dt}E_{x_0}(\xi_v)+2\gamma\|(-\Dx+1)^\frac{1}{4}\Dt v\|^2_{L^2(\Omega\cap B^1_{x_0})}\leq C\left(1+\|u\|^4_{L^{12}(\Omega\cap B^2_{x_0})}\right)E_{x_0}(\xi_v)+\\
\delta \|\phi_{\eb,x_0}(-\Dx+1)^\frac{1}{4}\Dt v\|^2_{L^2(\Omega)}+C_\delta\|\xi_v\|^2_{\Ee}, {\rm\ for\ all\ }t\geq 0,
\end{multline}
where $E_{x_0}(\xi_v)$ is defined by \eqref{Em} and $\eb$ is small. And repeating the arguments \eqref{4.12}-\eqref{4.13}
we derive
\begin{multline}
\label{5.4}
\|\xi_v(t)\|^2_{\E_{\{\mu,x_0\}}}+\kappa\int_0^t\|\phi_{\mu,x_0}(-\Dx+1)^\frac{1}{4}\Dt v\|^2_{L^2(\Omega)}\leq\\ Q_\mu(B_T)\(\|\xi_v(0)\|^2_{\E_{\{\mu,x_0\}}}+\int_0^t\|\xi_v(s)\|^2_{\E_{\{\mu,x_0\}}}ds\),\ t\in[0,T],
\end{multline}
where
\begin{equation}\label{BT}
B_T=\sup_{x_0\in\Omega}exp\left\{C\int_0^T\left(1+\|u\|^4_{L^{12}(\Omega\cap B^2_{x_0})}\right)dt\right\}
\end{equation}
for some constants $\kappa, C>0$ and monotone increasing function $Q_\mu$ with $\mu\in(0,\eb)$ which do not depend on $u$ and $t$. Taking the supremum with respect to $x_0$ in \eqref{5.4} and applying Gronwall's lemma we see that
\begin{equation}
\|\xi_v(t)\|^2_{\E_b}\leq Q(B_T)e^{Q(B_T)t}\|\xi_v(0)\|^2_{\E_b},\ t\in[0,T],
\end{equation}
for some monotone increasing function $Q$ that does not depend on $u$ and $t$. Taking into account \eqref{5.0}, \eqref{5.3}, \eqref{est adreg1}  we finish the proof. 
\end{proof}
\begin{rem}
Usually to prove the fact that $\xi_{\Dt u}(t)\in\E_b$ implies $\xi_u(t)\in\E^1_b$ it is more standard to multiply equation \eqref{5.2} by $\phi^2_{\eb,x_0}(-\Dx)u$ (or better by $\sum_k\d_k\phi^2_{\eb,x_0}\d_k u$, see \cite{Z sp.com} for details). This works well when, for example, $f'(s)\leq C(1+|s|^q)$ and $q\in [0,4)$ (that is in subcritical case) or  without growth restriction but with \eqref{f(s)s>-M} changed to $f'(s)\geq -M$. Since we have growth restriction (to have uniqueness)
and \eqref{f(s)s>-M} is less restrictive than $f'(s)\geq -M$ we have to deal in a different way, for example, using the trick with multiplication by $\phi^4_{\eb,x_0}u^3$. 
\end{rem}
The next theorem improves $L^4([0,T];L^{12}_{\{\eb,x_0\}}(\Omega))$ regularity obtained in Theorem \ref{th adreg1} to \linebreak
$L^2([0,T];\Cal H^\frac{3}{2}_{\{\eb,x_0\}}(\Omega))$ that will be used further to obtain smoothing property.
\begin{theorem}
\label{th adreg2}
Let assumptions of Theorem \ref{th ex} be satisfied and $\eb\in(0,\eb_0]$ be small enough. Then for every initial data $\xi_0=(u_0,u_1)\in\E_b$ the associated infinite-energy solution $u$ of problem \eqref{eq main} possesses extra regularity
$u\in L^2([0,T];\Cal H^\frac{3}{2}_{\{\eb,x_0\}}(\Omega))$ on arbitrary segment $[0,T]$ and the following estimate holds
\begin{equation}
\label{est adreg2}
\sup_{x_0\in \Omega}\int_{\max\{0,t-1\}}^t\|\phi_{\eb,x_0}(-\Dx+1)^\frac{3}{4}u(s)\|^2_{L^2(\Omega)}ds\leq Q_\eb(\|\xi_0\|_{\E_b})e^{-\beta t}+Q_\eb(\|g\|_{L^2_b(\Omega)}), \quad t\geq 0,
\end{equation}
for some constant $\beta>0$ and monotone increasing function $Q_\eb$ which are independent of $u$ ant $t$.
\end{theorem}
\begin{proof}
The result follows from multiplication of equation \eqref{eq main} by $\phi^2_{\e,x_0}(-\Dx+1)^\frac{1}{2}u$. As we already noticed in Theorem \ref{th uniq cont}, due to obtained extra regularity \eqref{est adreg1} we know that $f(u)\in L^1([0,T];L^2_{\{\eb,x_0\}}(\Omega))$ when $\eb$ is small enough. Thus, the product $(\phi^2_{\eb,x_0}(-\Dx+1)^\frac{1}{2}u,f(u))$ makes sense. To justify the product of the linear part of \eqref{eq main} by $\phi^2_{\e,x_0}(-\Dx+1)^\frac{1}{2}u$ one can approximate our solution $u$ by smooth solutions that can be done due to just obtained Theorem \ref{th sm1} and continuous dependence on initial data \eqref{cont.dep.id}. Since the approximation procedure is standard we focus on the derivation of the estimate only. Multiplication of \eqref{eq main} by $\phi^2_{\e,x_0}(-\Dx+1)^\frac{1}{2}u$ gives
\begin{multline}
\label{5.5}
\frac{d}{dt}\left((\phi_{\eb,x_0}\Dt u,\phi_{\eb,x_0}(-\Dx+1)^\frac{1}{2}u)+\frac{\gamma}{2}\|\phi_{\eb,x_0}(-\Dx+1)^\frac{1}{2}u\|^2_{L^2(\Omega)}\right)+\\
((-\Dx+1)u,\phi^2_{\eb,x_0}(-\Dx+1)^\frac{1}{2}u)=(1-\lambda_0)(\phi_{\eb,x_0}u,\phi_{\eb,x_0}(-\Dx+1)^\frac{1}{2}u)+\\
(\Dt u,\phi^2_{\eb,x_0}(-\Dx+1)^\frac{1}{2}\Dt u)-(\phi_{\eb,x_0}f(u),\phi_{\eb,x_0}(-\Dx+1)^\frac{1}{2}u)+(\phi_{\eb,x_0}g,\phi_{\eb,x_0}(-\Dx+1)^\frac{1}{2}u).
\end{multline} 
Similar to \eqref{4.91} the term involving fractional laplacian on the left hand side of \eqref{5.5} can be transformed as follows
\begin{multline}
((-\Dx+1)u,\phi^2_{\eb,x_0}(-\Dx+1)^\frac{1}{2}u)=\\
(\left[\phi_{\eb,x_0},(-\Dx+1)^\frac{1}{4}\right](-\Dx+1)^\frac{3}{4}u,\phi_{\eb,x_0}(-\Dx+1)^\frac{1}{2}u)+\\
(\phi_{\eb,x_0}(-\Dx+1)^\frac{3}{4}u,\left[(-\Dx+1)^\frac{1}{4},\phi_{\eb,x_0}\right](-\Dx+1)^\frac{1}{2}u)+\|\phi_{\eb,x_0}(-\Dx+1)^\frac{3}{4}u\|^2_{L^2(\Omega)},
\end{multline}    
and hence, using Proposition \ref{prop comest1} and fixing $\eb$ small enough, we get
\begin{multline}
\label{5.6}
((-\Dx+1)u,\phi^2_{\eb,x_0}(-\Dx+1)^\frac{1}{2}u)\geq\\
\frac{1}{2} \|\phi_{\eb,x_0}(-\Dx+1)^\frac{3}{4}u\|^2_{L^2(\Omega)}-C\|\phi_{\eb,x_0}(-\Dx+1)^\frac{1}{2}u\|^2_{L^2(\Omega)}.
\end{multline}
Analogously we have
\begin{multline}
(\Dt u,\phi^2_{\eb,x_0}(-\Dx+1)^\frac{1}{2}\Dt u)=(\phi_{\eb,x_0}\Dt u,\left[\phi_{\eb,x_0},(-\Dx+1)^\frac{1}{4}\right](-\Dx+1)^\frac{1}{4}\Dt u)+\\
(\left[(-\Dx+1)^\frac{1}{4},\phi_{\eb,x_0}\right]\Dt u,\phi_{\e,x_0}(-\Dx+1)^\frac{1}{4}\Dt u)+\|\phi_{\eb,x_0}(-\Dx+1)^\frac{1}{4}\Dt u\|^2_{L^2(\Omega)},
\end{multline}
and hence
\begin{equation}
\label{5.7}
(\Dt u,\phi^2_{\eb,x_0}(-\Dx+1)^\frac{1}{2}\Dt u)\leq C( \|\phi_{\eb,x_0}(-\Dx+1)^\frac{1}{4}\Dt u\|^2_{L^2(\Omega)}+\|\phi_{\eb,x_0}\Dt u\|^2_{L^2(\Omega)}).
\end{equation}
In addition, along the lines of \eqref{5.2} the non-linear term admits the estimate
\begin{multline}
\label{5.8}
|(\phi^2_{\eb,x_0}f(u),(-\Dx+1)^\frac{1}{2}u)|\leq \|\phi_{\eb,x_0}f(u)\|_{L^2(\Omega)}\|\phi_{\eb,x_0}(-\Dx+1)^\frac{1}{2}u\|_{L^2(\Omega)}\leq\\ 
 C_\eb\left(1+\|\phi_{\eb/5,x_0}u\|^5_{L^{10}(\Omega)}\right)\|\phi_{\eb,x_0}(-\Dx+1)^\frac{1}{2}u\|_{L^2(\Omega)}\leq\\
  C_\eb\(1+\|\xi_u\|^2_{\E_{\{\eb/5,x_0\}}}+\|\phi_{\eb/5,x_0}u\|^4_{L^{12}(\Omega)}\|\xi_u\|^2_{\E_{\{\eb/5,x_0\}}}\).
\end{multline}

Let us integrate \eqref{5.5} from $\max\{0,t-1\}$ to $t$. Keeping in mind \eqref{5.6}-\eqref{5.8}, applying dissipative estimate \eqref{est en} and extra regularity estimate \eqref{est adreg1} we obtain
\begin{equation}\label{5.8.1}
\int_{\max\{0,t-1\}}^t\|\phi_{\eb,x_0}(-\Dx+1)^\frac{3}{4}u(s)\|^2_{L^2(\Omega)}ds\leq Q_\eb(\|\xi_0\|_{\E_b})e^{-\beta t}+Q_\eb(\|g\|_{L^2_b(\Omega)}),\quad t\geq 0,
\end{equation} 
for some $\beta >0$ and monotone increasing function $Q_\eb$ which are independent of $u$. Since the estimate \eqref{5.8.1} is uniform with respect to $x_0$ we finish the proof.
\end{proof}
\begin{cor}
\label{cor.utt-sm}
Let assumptions of Theorem \ref{th ex} be satisfied and $\eb\in(0,\eb_0]$ be small enough. Then for every initial data $\xi_0=(u_0,u_1)\in\E_b$ the associated infinite-energy solution $u$ of problem \eqref{eq main} possesses extra regularity
\begin{equation}
(-\Dx+1)^{-\frac{1}{4}}\left(\phi_{\eb,x_0}\Dt^2u\right)\in L^2([0,T];L^2(\Omega))
\end{equation}
for arbitrary $T>0$ and the following estimate holds
\begin{multline}
\label{est.utt-sm}
\sup_{x_0\in \Omega}\int_{\max\{0,t-1\}}^t\|(-\Dx+1)^{-\frac{1}{4}}\left(\phi_{\eb,x_0}\Dt^2u(s)\right)\|^2_{L^2(\Omega)}ds\leq\\
 Q_\eb(\|\xi_0\|_{\E_b})e^{-\beta t}+Q_\eb(\|g\|_{L^2_b(\Omega)}), \quad t\geq 0,
\end{multline}
for some constant $\beta>0$ and monotone increasing function $Q_\eb$ which are independent of $u$ ant $t$.
\end{cor}
\begin{proof}
Since $u$ solves the equation \eqref{eq main} the formula for $\Dt^2 u$ reads
\begin{equation}
\Dt^2u=-\gamma(-\Dx+1)^\frac{1}{2}\Dt u-(-\Dx+1)u+(1-\lambda_0)u-f(u)+g.
\end{equation}
Let us estimate each of the terms.

By virtue of \eqref{f.growth}, continuous embedding $L^\frac{3}{2}(\Omega)\subset\Cal H^{-\frac{1}{2}}_{\{0,0\}}(\Omega)$ and interpolation $\left[L^6(\Omega),L^{12}(\Omega)\right]_{\frac{2}{5}}=L^\frac{15}{2}(\Omega)$ we have 
\begin{multline}
\int_{\max\{0,t-1\}}^t\|(-\Dx+1)^{-\frac{1}{4}}\left(\phi_{\eb,x_0}f(u)\right)\|^2_{L^2(\Omega)}ds\leq C\int_{\max\{0,t-1\}}^t\|\phi_{\eb,x_0}f(u(s))\|^2_{L^\frac{3}{2}(\Omega)}ds\leq\\
C_\eb+C\int_{\max\{0,t-1\}}^t\|\phi^5_{\eb/5,\,x_0}u^5\|^2_{L^\frac{3}{2}(\Omega)}ds\leq C_\eb+C\int_{\max\{0,t-1\}}^t\|\phi_{\eb/5,\,x_0}u\|^6_{L^6(\Omega)}\|\phi_{\eb/5,\,x_0}u\|^4_{L^{12}(\Omega)}ds.
\end{multline}

The term with fractional laplacian can be estimated using arguments on test functions and commutator estimates. Let $\chi\in L^2([0,T];L^2(\Omega))$ be a test function, then  
\begin{multline}
\int_{\max\{0,t-1\}}^t\left((-\Dx+1)^{-\frac{1}{4}}\left(\phi_{\eb,x_0}(-\Dx+1)^\frac{1}{2}\Dt u(s)\right),\chi(s)\right)ds=\\
\int_{\max\{0,t-1\}}^t\left(\phi_{\eb,x_0}(-\Dx+1)^\frac{1}{2}\Dt u(s),(-\Dx+1)^{-\frac{1}{4}}\chi(s)\right)ds=\\
\int_{\max\{0,t-1\}}^t\left(\left[\phi_{\eb,x_0},(-\Dx+1)^\frac{1}{4}\right](-\Dx+1)^\frac{1}{4}\Dt u(s),(-\Dx+1)^{-\frac{1}{4}}\chi(s)\right)ds+\\
\int_{\max\{0,t-1\}}^t\left(\phi_{\eb,x_0}(-\Dx+1)^\frac{1}{4}\Dt u(s),\chi(s)\right)ds\leq\\C\|\phi_{\eb,x_0}(-\Dx+1)^\frac{1}{4}\Dt u\|_{L^2([\max\{0,t-1\},t];L^2(\Omega))}\|\chi\|_{L^2([\max\{0,t-1\},t];L^2(\Omega))},
\end{multline} 
for some absolute constant $C$ assuming that $\eb$ is small, and therefore
\begin{multline}
\int_{\max\{0,t-1\}}^t\|(-\Dx+1)^{-\frac{1}{4}}\left(\phi_{\eb,x_0}(-\Dx+1)^\frac{1}{2}\Dt u(s)\right)\|^2_{L^2(\Omega)}ds\leq\\
 C\int_{\max\{0,t-1\}}^t\|\phi_{\eb,x_0}(-\Dx+1)^\frac{1}{4}\Dt u(s)\|^2_{L^2(\Omega)}ds.
\end{multline}
Similarly we obtain the estimate
\begin{multline}
\int_{\max\{0,t-1\}}^t\|(-\Dx+1)^{-\frac{1}{4}}\left(\phi_{\eb,x_0}(-\Dx+1)u(s)\right)\|^2_{L^2(\Omega)}ds\leq\\ 
C\int_{\max\{0,t-1\}}^t\|\phi_{\eb,x_0}(-\Dx+1)^\frac{3}{4}u(s)\|^2_{L^2(\Omega)}ds.
\end{multline}
The estimates of the rest terms are trivial. Combining the above estimates with energy estimate \eqref{est en}, Theorem \ref{th adreg1}, Theorem \ref{th adreg2} we complete the proof.
\end{proof}
The next theorem gives us smoothing property mentioned before.
\begin{theorem}
\label{th sm2}
Let assumptions of Theorem \ref{th ex} hold true, then for any initial data $\xi_0:=(u_0,u_1)\in \E_b$ the associated infinite-energy solution of problem \eqref{eq main} possesses smoothing property for $t\in(0,1]$:
\begin{equation}
\label{est sm2}
t^2\(\|\xi_{\Dt u}(t)\|^2_{\E_b}+\|\xi_u(t)\|^2_{\E^1_b}\)\leq Q(\|\xi_0\|_{\E_b})+Q(\|g\|_{L^2_b(\Omega)}),\ t\in(0,1],
\end{equation}
for some monotone increasing $Q$ which is independent of $t$ and $u$.
\end{theorem}
\begin{proof}
As in Theorem \ref{th sm1} we give only formal derivation of estimate \eqref{est sm2}. Setting $v=\Dt u$ as in Theorem \ref{th sm1} and multiplying inequality \eqref{5.3.1} by $t^2$ we obtain
\begin{multline}
\label{5.9}
\frac{d}{dt}\(t^2E_{x_0}(\xi_v)\)+2\gamma t^2\|(-\Dx+1)^\frac{1}{4}\Dt v\|^2_{L^2(\Omega\cap B^1_{x_0})}\leq\\ 
C\(1+\|u\|^4_{L^{12}(\Omega\cap B^2_{x_0})}\)t^2E_{x_0}(\xi_v)+\delta t^2\|\phi_{\eb,x_0}(-\Dx+1)^\frac{1}{4}\Dt v\|^2_{L^2(\Omega)}+C_\delta\|\xi_v\|^2_{\Ee}+ \\
2t\(\|\psi_{x_0}\Dt v\|^2_{L^2(\Omega)}+\lambda_0\|\psi_{x_0}v\|^2_{L^2(\Omega)}+\|\psi_{x_0}\nabla v\|^2_{L^2(\Omega)}\).
\end{multline}

Let us estimate the last three terms of \eqref{5.9} separately. The first one can be controlled due to commutator estimates as follows (assuming $\eb$ is small)
\begin{multline}
\label{5.10}
2t\|\psi_{x_0}\Dt v\|^2_{L^2(\Omega)}\leq C t\|\phi_{\eb,x_0}\Dt v\|^2_{L^2(\Omega)}= Ct(\phi_{\eb,x_0}\Dt^2 u,\phi_{\eb,x_0}\Dt v)\leq\\
Ct\|(-\Dx+1)^{-\frac{1}{4}}\left(\phi_{\eb,x_0}\Dt^2 u\right)\|_{L^2(\Omega)}\|\phi_{\eb,x_0}(-\Dx+1)^\frac{1}{4}\Dt v\|_{L^2(\Omega)}\leq\\
 \frac{t^2\delta}{2}\|\phi_{\eb,x_0}(-\Dx+1)^\frac{1}{4}\Dt v\|^2_{L^2(\Omega)}+C_\delta\|(-\Dx+1)^{-\frac{1}{4}}\left(\phi_{\eb,x_0}\Dt^2 u\right)\|^2_{L^2(\Omega)},
\end{multline}
where the term $\|(-\Dx+1)^{-\frac{1}{4}}\left(\phi_{\eb,x_0}\Dt^2 u\right)\|^2_{L^2(\Omega)}$ is not dangerous by Corollary \ref{cor.utt-sm}. 

The two last terms can be transformed as follows
\begin{equation}
\label{5.11}
2t\(\lambda_0\|\psi_{x_0}v\|^2_{L^2(\Omega)}+\|\psi_{x_0}\nabla v\|^2_{L^2(\Omega)}\)=
2\frac{d}{dt}\Big( \lambda_0t(\psi_{x_0}u,\psi_{x_0}v)+t(\psi_{x_0}\nabla u,\psi_{x_0}\nabla v) \Big)-2R,
\end{equation}
where $R$ is given by
\begin{equation}
R=\lambda_0(\psi_{x_0}u,\psi_{x_0}v)+(\psi_{x_0}\nabla u,\psi_{x_0}\nabla v)+\\
\lambda_0t(\psi_{x_0}u,\psi_{x_0}\Dt v)+t(\psi_{x_0}\nabla u,\psi_{x_0}\nabla \Dt v).
\end{equation}
Performing integration by parts in the previous formula and using commutator estimates it is easy to check the next inequality (at this point $\Cal H^\frac{3}{2}_{\{\eb,x_0\}}(\Omega)$-norm, given by Theorem \ref{th adreg2}, arises)
\begin{multline}
\label{5.12}
|R|\leq \frac{t^2\delta}{2}\|\phi_{\eb,x_0}(-\Dx+1)^\frac{1}{4}\Dt v\|^2_{L^2(\Omega)}+\\
C_\delta\(\|\phi_{\eb,x_0}(-\Dx+1)^\frac{3}{4}u\|^2_{L^2(\Omega)}+\|\phi_{\eb,x_0}(-\Dx+1)^\frac{1}{4}\Dt u\|^2_{L^2(\Omega)}\).
\end{multline}

Thus summarizing \eqref{5.9},\eqref{5.10}, \eqref{5.11} and \eqref{5.12} we get
\begin{multline}
\label{5.13}
\frac{d}{dt}\left(t^2E_{x_0}(\xi_v)- 2\lambda_0t(\psi_{x_0}u,\psi_{x_0}v)-2t(\psi_{x_0}\nabla u,\psi_{x_0}\nabla v)\right)+2\gamma t^2\|(-\Dx+1)^\frac{1}{4}\Dt v\|^2_{L^2(\Omega\cap B^1_{x_0})}\leq\\
C\left(1+\|u\|^4_{L^{12}(\Omega\cap B^2_{x_0})}\right)t^2E_{x_0}(\xi_v)+2\delta t^2\|\phi_{\eb,x_0}(-\Dx+1)^\frac{1}{4}\Dt v\|^2_{L^2(\Omega)}+\\
 C_\delta\(\|\phi_{\eb,x_0}(-\Dx+1)^\frac{3}{4}u\|^2_{L^2(\Omega)}+\|\phi_{\eb,x_0}(-\Dx+1)^\frac{1}{4}\Dt u\|^2_{L^2(\Omega)}+\|(-\Dx+1)^{-\frac{1}{4}}\left(\phi_{\eb,x_0}\Dt^2 u\right)\|^2_{L^2(\Omega)}\).
\end{multline}
Taking into account that
\begin{equation}
|2\lambda_0t(\psi_{x_0}u,\psi_{x_0}v)+2t(\psi_{x_0}\nabla u,\psi_{x_0}\nabla v)|\leq \frac{t^2}{2} E_{x_0}(\xi_v)+C\|\xi_u\|^2_{\Elr},
\end{equation} 
integrating \eqref{5.13} from $0$ to $t\in(0,1]$, using energy estimate \eqref{est en}, and extra regularity \eqref{est adreg2}, \eqref{est.utt-sm} we obtain
\begin{multline}\label{5.14}
t^2E_{x_0}(\xi_v(t))+4\gamma\int_0^ts^2\|(-\Dx+1)^\frac{1}{4}\Dt v(s)\|^2_{L^2(\Omega\cap B^1_{x_0})}ds\leq\\
C\int_0^t\left(1+\|u(s)\|^4_{L^{12}(\Omega\cap B^2_{x_0})}\right)s^2E_{x_0}(\xi_v(s))ds+4\delta \int_0^ ts^2\|\phi_{\eb,x_0}(-\Dx+1)^\frac{1}{4}\Dt v(s)\|^2_{L^2(\Omega)}ds+\\
C_\delta\(Q_{\eb}(\|\xi_0\|_{\E_b})+Q_{\eb}(\|g\|_{L^2_b(\Omega)})\),\ \rm{for\ all}\ t\in(0,1].
\end{multline}
Applying Gronwall's lemma in integral form to \eqref{5.14} we obtain
\begin{multline}\label{5.15}
t^2E_{x_0}(\xi_v(t))+4\gamma\int_0^ts^2\|(-\Dx+1)^\frac{1}{4}\Dt v(s)\|^2_{L^2(\Omega\cap B^1_{x_0})}ds\leq\\ 
B_1\(4\delta \int_0^ ts^2\|\phi_{\eb,x_0}(-\Dx+1)^\frac{1}{4}\Dt v(s)\|^2_{L^2(\Omega)}ds+
C_\delta \(Q_{\eb}(\|\xi_0\|_{\E_b})+Q_{\eb}(\|g\|_{L^2_b(\Omega)})\)\),\ t\in(0,1], 
\end{multline}
where 
\begin{equation}
B_1=C\sup_{x_0\in\Omega}\int_0^1\(1+\|u(s)\|^4_{L^{12}(\Omega\cap B^2_{x_0})}\)ds.
\end{equation}
Multiplying the estimate \eqref{5.15} by $\phi^2_{\mu,x_0}$ with $\mu\in(0,\eb)$, integrating the obtained inequality with respect to $x_0$ over $\Omega$, applying Theorem \ref{th wsp2} on the left and Theorem \ref{th wsp1} on the right, choosing $\delta=\delta(B_1)$ small enough we find
\begin{equation}
t^2\|\xi_v(t)\|^2_{\E_{\{\mu,x_0\}}}\leq Q(B_1)(Q(\|\xi_0\|_{\E_b})+Q(\|g\|_{L^2_b(\Omega)})), t\in(0,1],
\end{equation}
for some increasing monotone $Q$. Taking supremum with respect to $x_0$ in the last estimate, using Corollary \ref{cor wsp4} and Theorem \ref{th adreg1} we finish the proof.  
\end{proof}

The next corollary improves estimate \eqref{est sm1} to its dissipative version.
\begin{cor}
\label{cor sm3}
Let the assumptions of Theorem \ref{th ex} hold. And let $u$ be an infinite-energy solution of problem \eqref{eq main} with initial data $\xi_0\in\E^1_b$. Then $u$ admits the following estimate
\begin{equation}
\label{est sm3}
\|\xi_u(t)\|_{\E^1_b}+\|\xi_{\Dt u}(t)\|_{\E_b}\leq Q(\|\xi_0\|_{\E^1_b})e^{-\beta t}+Q(\|g\|_{L^2_b(\Omega)}),\quad t\geq 0,
\end{equation} 
for some $\beta>0$ and increasing function $Q$ which are independent of $t$ and $u$.
\end{cor}
\begin{proof}
Indeed, by virtue of Theorem \ref{th sm2} we have
\begin{equation}
\|\xi_u(t+1)\|_{\E^1_b}+\|\xi_{\Dt u}(t+1)\|_{\E_b}\leq Q(\|\xi_u(t)\|_{\E_b})+Q(\|g\|_{b,0,2}),\quad t\geq 0,
\end{equation}
that together with dissipative estimate \eqref{est en} and technical Lemma \ref{lem Q(Ae-t+B)<Q(A)e-t+Q(B)} (see below) gives the desired estimate \eqref{est sm3} for $t\geq 1$. Also, obviously, Theorem \ref{th sm1} gives estimate \eqref{est sm3} on segment $[0,1]$, so the proof is compete.
\end{proof}

The following lemma, used in Corollary \ref{cor sm3}, was originally proved in 
\cite{vz96}.
\begin{lemma}
\label{lem Q(Ae-t+B)<Q(A)e-t+Q(B)}
Let $Q:\R_+\to \R_+$ be a smooth function, $L_1,L_2\in\R_+$ and $\alpha>0$. Then there exists a monotone increasing function $Q_1:\R_+\to\R_+$ such that
\begin{equation}
Q(L_1+L_2e^{-\alpha t})\leq Q_1(L_1)+Q_1(L_2)e^{-\alpha t}.
\end{equation}
\end{lemma}

\section{The Attractors}
\label{s.a}
In this section we establish existence of locally compact attractor for the natural dynamical system $(S_t,\E_b)$ associated with problem \eqref{eq main}:
\begin{equation}
\label{def St}
S_t:\E_b\to\E_b,\ \quad S_t\xi_0=\xi_u(t),  
\end{equation}
where $\xi_0\in\E_b$ and $\xi_u(t)$ is infinite-energy solution of problem \eqref{eq main} with initial data $\xi_0$. Locally compact attractors play central
role in capturing long-time behaviour of PDEs in unbounded domains. These objects are natural counterparts of global attractors which are commonly used in studying asymptotic behaviour of PDEs in bounded domains (see \cite{bk BV}).

Let us recall the definition of locally compact attractor (see\cite{MZ Dafer 2008}) 
\begin{Def}
\label{def A}
A set $\A\subset \E_b$ to be called locally compact attractor (further just attractor for brevity) of the dynamical system $(S_t,\E_b)$ in $\E_b$ iff
\par
(i) The set $\A$ is bounded in $\E_b$ and compact in $\E_{loc}:=H^1_{loc}(\Omega)\times L^2_{loc}(\Omega)$.
\par
(ii) The set $\A$ is strictly invariant, that is $S_t\A=\A$, for all $t\geq 0$.
\par
(iii) The  set $\A$ uniformly attracts any bounded subset of $\E_b$ in topology of $\E_{loc}$. That is, for any neighbourhood $\Cal{O}(\A)$ of $\A$ in topology of $\E_{loc}$ and for every bounded subset $B$ in topology of $\E_b$, there exists $T=T(\Cal{O}(\A),B)$ such that
\begin{equation}
S_tB\subset\Cal{O}(A),\quad t\geq T.
\end{equation}
\end{Def}
We remind that condition $(i)$ means that $\A|_{\Omega_1}$ is compact in $H^1(\Omega_1)\times L^2(\Omega_1)$ for any bounded domain $\Omega_1\subset \Omega$. Similarly, condition $(iii)$ means that for every bounded set $B$ in $\E_b$ and every $H^1(\Omega_1)\times L^2(\Omega_1)$-neighbourhood $\Cal{O}(\A|_{\Omega_1})$ of the restriction $\A|_{\Omega_1}$ there exists $T=T(\Omega_1,\Cal{O},B)$ such that:
\begin{equation}
(S_tB)|_{\Omega_1}\subset\Cal{O}(\A|_{\Omega_1}),\ t\geq T,
\end{equation}
where, again, $\Omega_1$ is an arbitrary bounded domain of $\Omega$. Also we would like to point out that definition \eqref{def A} coincides with definition of $(\E_b,\E_{loc})$-attractor in terminology of \cite{bk BV}. 
\begin{rem}
One may ask why we require compactness of  $\A$ as well as well as attraction property in the space $\E_{loc}$ and not in $\E_b$. In \cite{Z entropy} there was constructed an example that shows that, as a rule, when we consider dynamical system generated by a partial differential equation in unbounded domain then attractor can not be a compact subset of $\E_b$. Furthermore, even when attractor is a compact set of $\E_b$ it is too optimistic to expect attraction property in $\E_b$. The corresponding counter-example is constructed in \cite{Z mps}. One more key difference between the case of locally compact attractors in unbounded domains from attractors in bounded domains is that they usually have infinite fractal dimension. Example illustrating the nature of this fact can be found in \cite{MZ Dafer 2008}.
\end{rem}

The next theorem establishes existence of locally compact attractor for problem \eqref{eq main}. 
\begin{theorem}
\label{th attr}
Let assumptions of Theorem \ref{th ex} be satisfied. Then dynamical system $(S_t,\E_b)$ defined by \eqref{def St} possesses a locally compact attractor $\A$ in $\E_b$ which is a bounded subset of $\E^1_b$ can be described as
\begin{equation}
\label{A=K(0)}
\A=\Cal K|_{t=0},
\end{equation}
where $\Cal K$ is the set of all complete bounded in $\E_b$ solutions of \eqref{eq main}.
\end{theorem} 
\begin{proof}
According to abstract attractor's existence theorem (see \cite{bk BV}) we need to
check two points:
\par
$(i)$ dynamical system $(S_t,\E_b)$ possesses a bounded in $\E_b$ absorbing set $\Cal B_0$ which is compact in $\E_{loc}$. We remind that $\Cal B_0$ is absorbing means that for any bounded set $B\subset \E_b$ there exists time $T_B$ such that for all $t\geq T_B$ we have $S_tB\subset \Cal B_0$. 
\par
$(ii)$ for every fixed $t\geq 0$ the evolutionary operator $S_t$ is continuous on $\Cal B_0$ in topology of $\E_{loc}$.

The desired absorbing set can be constructed by standard procedure as follows.
From Theorem \ref{th sm2} and Corollary \ref{cor sm3} one can easily see
that closed ball $B^R_0$ in $\E^1_b$ with large enough radius $R$ is an absorbing ball. Since the embedding $\E^1_b\subset \E_{loc}$ is compact then $B^R_0$ is compact in $\E_{loc}$. Thus we can take $\Cal B_0=B^R_0$.

$\E_{loc}$-Continuity of $S_t$ on $\Cal B_0$ is a direct consequence of the continuity of $S_t$ in topology of $\E_b$ (see Theorem \ref{th uniq cont}).

Thus points $(i)$ and $(ii)$ are checked and hence the attractor exists. The boundedness of the attractor in $\E^1_b$ follows from the fact that $\Cal B_0$ is bounded in $\E^1_b$ by construction.

Finally representation \eqref{A=K(0)} is classical. Obviously, elements of $\A$ generate complete bounded in $\E_b$ trajectories, hence $\A\subset \Cal K|_{t=0}$. On the other hand, any complete \emph{bounded} in $\E_b$ trajectory $\Gamma=\{\xi_u(t)\}_{t\in\R}$ of the dynamical system $(S_t,\E_b)$ (see \eqref{def St}) gets into arbitrary small neighbourhood $\Cal O(\A)$ in $\E_{loc}$ of $\A$. But since trajectory $\Gamma$ is complete it is invariant, that is, the whole trajectory itself is in $\Cal O(\A)$. In addition, since $\A$ is $\E_{loc}$-closed we get that $\Gamma\subset \A$ and in particular $\Gamma|_{t=0}\subset \A$.
\end{proof}
\section{Appendix}\label{sec.Appendix}
%We start with non-dissipative version of Theorem \ref{th ex}, namely
The proof of Theorem \ref{th ex} is based on the following Gronwall-type lemma. In turn, the proof of the lemma rests on ideas presented in \cite{JZ,Z NSR2,G.NS14} (see also \cite{AZ NStrp,Zstr2007} for slightly more abstract variant of this approach and \cite{GPZ}, where the same idea is illustrated in a simpler case).
\begin{lemma}\label{lem.gr}
Let $y_{\sigma}(t),\ Y_{\sigma}(t):\R_+\to\R$ be two positive absolutely continuous functions depending on the parameter $\sigma\in\Sigma$ (here $\Sigma$ is an abstract non-empty set), satisfying the following differential inequality
\begin{equation}
\label{10.grineq}
\frac{d}{dt}Y_\sigma+\kappa\left(y_\sigma\right)^\frac{1}{p}\leq H,\ t\geq 0,
\end{equation}
with some constants $\kappa,\ H>0$ and $p\geq 1$ independent of $\sigma$. Let, in addition, $Y_\sigma$ and $y_\sigma$ be such that
\begin{align}
\label{10.equiv}
& \lambda\sup_{\sigma\in\Sigma} Y_\sigma(t)\leq\sup_{\sigma\in\Sigma} y_\sigma(t)\leq \Lambda\sup_{\sigma\in\Sigma} Y_\sigma(t),\ &\mbox{{\rm uniformly} w.r.t. }t\geq0;\\
\label{10.int.equiv}
& \lambda\sup_{\sigma\in\Sigma}\int_s^tY_\sigma(\tau)^\frac{1}{p}d\tau\leq\sup_{\sigma\in\Sigma}\int_s^ty_\sigma(\tau)^\frac{1}{p}d\tau\leq \Lambda\sup_{\sigma\in\Sigma}\int_s^tY_\sigma(\tau)^\frac{1}{p}d\tau, &\mbox{{\rm uniformly} w.r.t. }0\leq s<t;
\end{align}
for some strictly positive constants $\lambda,\ \Lambda$. Then the following estimate holds 
\begin{equation}
\label{10.a.disest}
Y(t)\leq Q(Y(0))e^{-\beta t}+Q(H),\ t\geq 0, 
\end{equation}
for some constant $\beta>0$ and increasing function $Q$, where $Y(t)=\sup_{\sigma\in\Sigma}Y_\sigma(t)$.
\end{lemma}
\begin{proof}
Integrating inequality \eqref{10.grineq} from $s$ to $t$ and taking $\frac{1}{p}$-root of both sides of the obtained inequality we have
\begin{equation}
\label{10.4}
Y_\sigma(t)^\frac{1}{p}+
\kappa^\frac{1}{p}\left(\int_s^ty_\sigma(\tau)^\frac{1}{p}\,d\tau\right)^\frac{1}{p}\leq 2Y_\sigma(s)^\frac{1}{p}+2H^\frac{1}{p}(t-s)^\frac{1}{p},\qquad {\rm for\ all\ } 0\leq s<t.
\end{equation}
Let $\{T_k\}_{k=0}^\infty$ be an increasing sequence of times going to $+\infty$ with $T_0=0$. We will specify this sequence below. Setting $t=T_{k+1}$, integrating the inequality \eqref{app.4} with respect to $s$ from $T_{k-1}$ to $T_k$, and skipping the first term on the left hand side we find 
\begin{multline}
\label{10.5}
\kappa^\frac{1}{p}(T_k-T_{k-1})\left(\int_{T_k}^{T_{k+1}}y_\sigma(\tau)^\frac{1}{p}\,d\tau\right)^\frac{1}{p}\leq\\
2\int_{T_{k-1}}^{T_k}Y_\sigma(s)^\frac{1}{p}\,ds+
2H^\frac{1}{p}(T_k-T_{k-1})(T_{k+1}-T_{k-1})^\frac{1}{p},\ k\in\mathbb{N}. 
\end{multline}
Thus, introducing notation
\begin{align}
\label{Vk}
&V(k):=\sup_{\sigma\in\Sigma}\int_{T_k}^{T_{k+1}}Y_\sigma(\tau)^\frac{1}{p}\,d\tau,\ k\in \mathbb{N}\cup\{0\},\\
\label{Wk}
&W(k):=\left(\frac{V(k)}{T_{k+1}-T_k}\right)^\frac{1}{p},\ k\in \mathbb{N}\cup\{0\},
\end{align}
and taking supremum in $\sigma$ we rewrite \eqref{10.5} in the form
\begin{equation}
\label{app.6}
W(k)\leq 2\left(\frac{L}{\lambda\kappa}\right)^\frac{1}{p}\frac{1}{(T_k-T_{k-1})^\frac{1}{p}}W(k-1)^p+2\left(\frac{HL}{\lambda\kappa}\right)^\frac{1}{p},\ k\in\mathbb{N},
\end{equation}
where we have assumed that the sequence $\{T_k\}_{k=1}^\infty$ satisfies the following condition
\begin{equation}
\label{Tk.0}
T_{k+1}-T_{k-1}\leq L(T_{k+1}-T_k),\ k\in\mathbb{N},
\end{equation}
with large enough constant $L\geq 1$ to be specified below. Furthermore, let us assume for a moment, that the sequence $\{T_k\}_{k=0}^\infty$ is such that 
\begin{equation}
\label{Tk.1}
2\left(\frac{L}{\lambda\kappa}\right)^\frac{1}{p}\frac{1}{(T_k-T_{k-1})^\frac{1}{p}}W(k-1)^{p-1}\leq\frac{1}{2},\ {\rm for\ all\ }k\in\mathbb{N}.
\end{equation}
Then \eqref{app.6} reads
\begin{equation}
\label{app.7}
W(k)\leq \frac{1}{2}W(k-1)+2\left(\frac{HL}{\lambda\kappa}\right)^\frac{1}{p},\ k\in\mathbb{N}.
\end{equation}
Iterating \eqref{app.7} we conclude that
\begin{equation}
\label{app.7.5}
W(k)\leq \left(\frac{1}{2}\right)^kW(0)+4\left(\frac{HL}{\lambda\kappa}\right)^\frac{1}{p},\ k\in\mathbb{N}.
\end{equation}
Moreover, from \eqref{10.4} we see, that 
\begin{equation}
W(0)\leq \frac{2}{(\lambda\kappa T_1)^\frac{1}{p}}Y(0)^\frac{1}{p}+4\left(\frac{H}{\lambda\kappa}\right)^\frac{1}{p},
\end{equation}
Therefore \eqref{app.7.5} gives
\begin{equation}
\label{M(k)}
W(k)\leq \left(\frac{1}{2}\right)^k\frac{2}{(\lambda\kappa T_1)^\frac{1}{p}}Y(0)^\frac{1}{p}+8\left(\frac{H(L+1)}{\lambda\kappa}\right)^\frac{1}{p}=:M(k),\ k\in\{0\}\cup\mathbb{N}.
\end{equation}

Now let us define the sequence $\{T_k\}_{k=0}^\infty$ as follows: $T_0$ equals to $0$ and for $k\in\mathbb{N}$ it solves the equation
\begin{equation}
2\left(\frac{L}{\lambda\kappa}\right)^\frac{1}{p}\frac{1}{(T_k-T_{k-1})^\frac{1}{p}}(M(k-1))^{p-1}=\frac{1}{2},\ k\in\mathbb{N},
\end{equation} 
that is
\begin{equation}
\label{Tk.rec}
T_k=T_{k-1}+4^p\frac{L}{\lambda\kappa}(M(k-1))^{p(p-1)},\ k\in \mathbb{N},
\end{equation}
where $M(k)$ is given by \eqref{M(k)}. It is easy to see that equality \eqref{Tk.rec} indeed uniquely defines an increasing positive sequence $\{T_k\}_{k=0}^\infty$ (a bit more care is required for $k=1$). Let us check that such defined sequence satisfies \eqref{Tk.0} if $L$ is large. Clearly, for our choice of $T_k$ the inequality \eqref{Tk.0} takes the form 
\begin{equation}
\label{L.ch}
M(k-1)\leq \sqrt[p(p-1)]{L-1}M(k),\ k\in\mathbb{N}.
\end{equation} 
Since $M(k)$ satisfies the following recurrence relation
\begin{equation}
\label{M.rec}
M(k)=\frac{1}{2}M(k-1)+4\left(\frac{H(L+1)}{\lambda\kappa}\right)^\frac{1}{p},\ k\in\mathbb{N},
\end{equation}
we see that \eqref{L.ch} is valid, for example, for $L=2^{p(p-1)}+1$.

Let us verify that for our choice of $T_k$ the inequality \eqref{M(k)} holds true. Indeed, due to \eqref{app.6}, \eqref{M.rec}, we have
\begin{multline}
\label{W<M}
W(k)-M(k)\leq 2\left(\frac{L}{\lambda\kappa}\right)^\frac{1}{p}\frac{1}{(T_k-T_{k-1})^\frac{1}{p}}\left(W(k-1)\right)^p-\frac{1}{2}M(k-1)\leq\\
\left(2\left(\frac{L}{\lambda\kappa}\right)^\frac{1}{p}\frac{1}{(T_k-T_{k-1})^\frac{1}{p}}\left(M(k-1)\right)^2-\frac{1}{2}\right)M(k-1)=0,\ k\in\mathbb{N}, 
\end{multline} 
by induction.

Coming back to $V(k)$, and taking into account \eqref{Tk.rec}, \eqref{M(k)} we conclude that for any $Y(0)$ such that $Y(0)\leq R$ there exists time $T_R$ such that for all $t\geq T_R$ we have
\begin{equation}
\label{app.8}
\sup_{\sigma\in\Sigma}\int_{t}^{t+1}Y_\sigma(\tau)^\frac{1}{p}\,d\tau\leq Q_1(H),\ t\geq T_R, 
\end{equation}
for some monotone increasing function $Q_1$. Furthermore, assuming that $t\geq T_R+1$, integrating \eqref{10.4} with respect to $s$ from $t-1$ to $t$, taking into account \eqref{app.8} we derive
\begin{equation}
\label{app.9}
Y(t)\leq Q_2(H),\ t\geq T_R+1,
\end{equation}
for a monotone increasing function $Q_2$. It is easy to see that \eqref{app.9} yields the desired estimate \eqref{10.a.disest} for some monotone increasing $Q$ and constant $\beta>0$ that completes the proof.
\end{proof}
\begin{rem}
Let assumptions of Lemma \ref{lem.gr} be satisfied with $p>1$ but $H=0$. Then there exists a finite time $T_*= T_*(Y(0))<+\infty$ such that 
\begin{equation}
\label{Y=0}
Y(t)\equiv0,\qquad t\geq T_*.
\end{equation}
\end{rem}
\begin{proof}
Indeed, repeating the arguments of Lemma \ref{lem.gr} we, as previously, conclude that
\begin{equation}
W(k)\leq M(k).
\end{equation}
However, since $H=0$, we see that
\begin{equation}
\lim_{k\to+\infty}T_k=CY(0)^{p-1}:=T_*,
\end{equation} 
for some positive constant $C=C(p,\kappa)$. Since $Y_\sigma(t)$ is continuous we have
\begin{equation}
\lim_{k\to+\infty}\frac{\int_{T_k}^{T_{k+1}}Y_\sigma(s)ds}{T_{k+1}-T_k}=Y_\sigma(T_*).
\end{equation}
On the other hand, due to \eqref{M(k)} and the fact that $H=0$ we have $\lim_{k\to\infty} W(k)=0$. Thus $Y_\sigma(T_*)=0$ for all $\sigma\in\Sigma$. Since both $y_\sigma$ and $Y_\sigma$ are positive integrating \eqref{10.grineq} from $T_*$ to $T>T_*$ we obtain the desired result. 
\end{proof}
We proceed to the proof of Theorem \ref{th ex}.
\begin{proof}[proof of Theorem \ref{th ex}]
We give the formal derivation of the apriori estimate \eqref{est en}, highlighting neat points related to the work with uniformly local and weighted spaces and fractional laplacian. The existence of the solutions follows from this estimate and the fact that infinite-energy solutions can be approximated by finite-energy solutions. Multiplying the equation \eqref{eq main} by $\phi^2_{\eb,x_0}\Dt u+\delta\phi^2_{\eb,x_0}u$ with small enough $\delta>0$ (to be fixed below) we get 
\begin{multline}
\label{app.0}
\frac{1}{2}\frac{d}{dt}E_{\eb,x_0}(\xi_u)+\gamma\|\phi_{\eb,x_0}(-\Dx+1)^\frac{1}{4}\Dt u\|^2_{L^2(\Omega)}+\\
\delta\left(\|\phi_{\eb,x_0}\nabla u\|^2_{L^2(\Omega)}+\lambda_0\|\phi_{\eb,x_0}u\|^2_{L^2(\Omega)}\right)+\delta(\phi^2_{\eb,x_0},f(u)u)=H(t),
\end{multline}
where we have used the following manipulations with the terms involving fractional laplacian
\begin{multline}
\delta\gamma(\phi_{\eb,x_0}(-\Dx+1)^\frac{1}{2}\Dt u,\phi_{\eb,x_0}u)=\delta\gamma([\phi_{\eb,x_0},(-\Dx+1)^\frac{1}{4}](-\Dx+1)^\frac{1}{4}\Dt u,\phi_{\eb,x_0}u)+\\
\frac{\delta\gamma}{2}\frac{d}{dt}\|\phi_{\eb,x_0}(-\Dx+1)^\frac{1}{4}u\|^2_{L^2(\Omega)}+\delta\gamma(\phi_{\eb,x_0}(-\Dx+1)^\frac{1}{4}\Dt u, [(-\Dx+1)^\frac{1}{4},\phi_{\eb,x_0}]u),
\end{multline}
\begin{multline}
\gamma(\phi_{\eb,x_0}(-\Dx+1)^\frac{1}{2}\Dt u,\phi_{\eb,x_0}\Dt u)=\gamma([\phi_{\eb,x_0},(-\Dx+1)^\frac{1}{4}](-\Dx+1)^\frac{1}{4}\Dt u,\phi_{\eb,x_0}\Dt u)+\\
\gamma\|\phi_{\eb,x_0}(-\Dx+1)^\frac{1}{4}\Dt u\|^2_{L^2(\Omega)}+\gamma(\phi_{\eb,x_0}(-\Dx+1)^\frac{1}{4}\Dt u, [(-\Dx+1)^\frac{1}{4},\phi_{\eb,x_0}]\Dt u),
\end{multline}
and use notations $F(u):=\int_0^uf(v)dv$,
\begin{multline}
E_{\eb,x_0}(\xi_u):=\|\xi_u\|^2_{\Ee}+2(\phi^2_{\eb,x_0},F(u))-2(\phi_{\eb,x_0}g,\phi_{\eb,x_0}u)+\\
2\delta(\phi_{\eb,x_0}\Dt u,\phi_{\eb,x_0}u)+\delta\gamma\|\phi_{\eb,x_0}(-\Dx+1)^\frac{1}{4}u\|^2_{L^2(\Omega)},
\end{multline}
and the right hand side $H(t)$ is given by
\begin{align}
&H(t)=H_1(t)+H_2(t)+H_3(t),\\
&H_1(t)=\delta(\phi_{\eb,x_0}g,\phi_{\eb,x_0}u)+\delta\|\phi_{\eb,x_0}\Dt u\|^2_{L^2(\Omega)}-2(\phi_{\eb,x_0}\nabla\phi_{\eb,x_0}\Dt u,\nabla u)-\\
&\phantom{/////////////////////////////////////////////////////}2\delta(\phi_{\eb,x_0}\nabla\phi_{\eb,x_0},u\nabla u),\notag\\
&H_2(t)=-\delta\gamma([\phi_{\eb,x_0},(-\Dx+1)^\frac{1}{4}](-\Dx+1)^\frac{1}{4}\Dt u,\phi_{\eb,x_0}u)-\\
&\phantom{////////////////////////////////}\delta\gamma(\phi_{\eb,x_0}(-\Dx+1)^\frac{1}{4}\Dt u,[(-\Dx+1)^\frac{1}{4},\phi_{\eb,x_0}]u),\notag\\
&H_3(t)=-\gamma([\phi_{\eb,x_0},(-\Dx+1)^\frac{1}{4}](-\Dx+1)^\frac{1}{4}\Dt u,\phi_{\eb,x_0}\Dt u)-\\
&\phantom{////////////////////////////////}\gamma(\phi_{\eb,x_0}(-\Dx+1)^\frac{1}{4}\Dt u,[(-\Dx+1)^\frac{1}{4},\phi_{\eb,x_0}]\Dt u)\notag.
\end{align}
Let us estimate the right hand side $H(t)$. Choosing $\delta=\delta(\gamma)$ and $\eb=\eb(\delta)$ small it is easy to see that $H_1$-term admits the estimate
\begin{multline}
\label{app.H1-est}
|H_1|\leq C\|\phi_{\eb,x_0}g\|^2_{L^2(\Omega)}+\frac{\gamma}{6}\|\phi_{\eb,x_0}(-\Dx+1)^\frac{1}{2}\Dt u\|_{L^2(\Omega)}+\\
\frac{\delta}{4}\left(\|\phi_{\eb,x_0}\nabla u\|^2_{L^2(\Omega)}+\lambda_0\|\phi_{\eb,x_0}u\|^2_{L^2(\Omega)}\right).
\end{multline} 
Also, due to Proposition \ref{prop comest1} about the commutator estimates, we find
\begin{align}
\label{app.H2-est}
&|H_2|\leq \frac{\gamma}{6}\|\phi_{\eb,x_0}(-\Dx+1)^\frac{1}{2}\Dt u\|_{L^2(\Omega)}+\frac{\delta}{4}\left(\|\phi_{\eb,x_0}\nabla u\|^2_{L^2(\Omega)}+\lambda_0\|\phi_{\eb,x_0}u\|^2_{L^2(\Omega)}\right),\\
\label{app.H3-est}
&|H_3|\leq \frac{\gamma}{6}\|\phi_{\eb,x_0}(-\Dx+1)^\frac{1}{2}\Dt u\|^2_{L^2(\Omega)}.
\end{align} 
Combining \eqref{app.0}, \eqref{app.H1-est}-\eqref{app.H3-est}, \eqref{f(s)s>-M} and due to the fact that $\delta=\delta(\gamma)$ is small we obtain
\begin{equation}
\label{app.1}
\frac{1}{2}\frac{d}{dt}E_{\eb,x_0}(\xi_u)+\frac{\gamma}{4}\|\phi_{\eb,x_0}(-\Dx+1)^\frac{1}{4}\Dt u\|^2_{L^2(\Omega)}+
\frac{\delta}{4}\|\xi_u\|^2_{\Ee}\leq C_{\eb}(1+\|\phi_{\eb,x_0}g\|^2_{L^2(\Omega)}).
\end{equation}
Now let us estimate $E_{\eb,x_0}(\xi_u)$. Dissipative assumption \eqref{f(s)s>-M} implies that 
\begin{equation}
\label{F(u)>-au^2-C}
F(u)\geq -\frac{\lambda_0}{8}u^2-C,\ {\rm for\ all\ }u\in\R,
\end{equation}
for some absolute constant $C$ (depending on $\lambda_0$). Hence, choosing $\delta=\delta(\gamma,\lambda_0)$ even smaller if needed and choosing constant $C_\eb$ sufficiently large, we get
\begin{equation}
\label{xi from ab}
\mathbb{E}_{\eb,x_0}(\xi_u):=E_{\eb,x_0}(\xi_u)+C_{\eb}\left(1+\|\phi_{\eb,x_0}g\|^ 2_{L^2(\Omega)}\right)\geq\frac{1}{2}\|\xi_u\|^2_{\Ee}.
\end{equation}
On the other hand, for small $\eb>0$, we have
\begin{equation}
\label{xi from b}
\|\xi_u\|^2_{\Ee}\geq c\left(\mathbb{E}_{3\eb,x_0}(\xi_u)\right)^\frac{1}{3}-C_\eb\(1+\|\phi_{3\eb,x_0}g\|^ 2_{L^2(\Omega)}\), 
\end{equation}
for some positive absolute constant $c$ and positive constant $C_\eb$.
Indeed, using growth assumption \eqref{f.growth} and continuous embedding $ H^1_0(\Omega)\subset L^6(\Omega)$, we find
\begin{multline}
\mathbb{E}_{\eb,x_0}(\xi_u)\leq C\|\xi_u\|^2_{\Ee}+C_\eb\left(1+\|\phi_{\eb,x_0}g\|^2_{L^2(\Omega)}\right)+C(\phi^2_{\eb,x_0},1+|u|^6)\leq\\
C\|\xi_u\|^2_{\Ee}+C_\eb\left(1+\|\phi_{\eb,x_0}g\|^2_{L^2(\Omega)}\right)+C\|\phi_{\eb/3,x_0}u\|^6_{L^6(\Omega)}\leq\\
C\|\xi_u\|^6_{\E_{\{\eb/3,x_0\}}}+C_\eb\left(1+\|\phi_{\eb,x_0}g\|^2_{L^2(\Omega)}\right),
\end{multline}
for some absolute constant $C$ as long as $\eb=\eb(\lambda_0)$ is small. Taking the cubic root of the above inequality we end up with \eqref{xi from b}.

Thus, from \eqref{app.1}, \eqref{xi from ab}, \eqref{xi from b} we obtain
\begin{equation}
\label{app.3}
\frac{d}{dt}\mathbb{E}_{\eb,x_0}(\xi_u)+\kappa\left(\mathbb{E}_{3\eb,x_0}(\xi_u)\right)^\frac{1}{3}+\kappa\|\phi_{\eb,x_0}(-\Dx+1)^\frac{1}{4}\Dt u\|^2_{L^2(\Omega)}\leq C_\eb\left(1+\|g\|^ 2_{L^2_b(\Omega)}\right),
\end{equation} 
uniformly with respect to $x_0\in\Omega$, where $\kappa=\kappa(\gamma,\lambda_0)>0$ is small enough absolute constant. Also from this moment we suppose that $\eb$ is fixed.

We claim that inequality \eqref{app.3} implies the desired estimate \eqref{est en}. To this end we would like to apply Lemma \ref{lem.gr}. Let us denote 
\begin{equation}
Y_{x_0}(t)=\mathbb{E}_{\eb,x_0}(\xi_u(t)),\qquad y_{x_0}(t)=\mathbb{E}_{3\eb,x_0}(\xi_u(t)).
\end{equation}
So we only need to check conditions \eqref{10.equiv}, \eqref{10.int.equiv}. Completing the squares and taking into account \eqref{F(u)>-au^2-C} we see that that the right parts of the inequalities \eqref{10.equiv}, \eqref{10.int.equiv} are fulfilled with $\Lambda =1$ (they hold also without supremum in $x_0$) and the left parts of the inequalities \eqref{10.equiv}, \eqref{10.int.equiv} are valid exactly due to supremum in $x_0$ for some $\lambda =\lambda(\eb)$. Therefore by Lemma \ref{lem.gr} and \eqref{xi from ab} we obtain
\begin{equation}
\label{app.4}
\|\xi_u(t)\|^2_{\E_b}\leq Q_\eb(\|\xi_0\|_{\E_b})e^{-\beta t}+Q_\eb(\|g\|_{L^2_b(\Omega)}),\ t\geq 0,
\end{equation}
for some monotone increasing function $Q_\eb$ and $\beta>0$ independent of solution $u$. In addition, integrating \eqref{app.3} from $\max\{0,t-1\}$ to $t$, taking supremum in $x_0$ and using \eqref{app.4} we get 
\begin{equation}
\sup_{x_0\in\Omega}\int_{\max\{0,t-1\}}^t\|\phi_{\eb,x_0}(-\Dx+1)^\frac{1}{4}\Dt u(s)\|^2_{L^2(\Omega)}ds\leq\\
 Q_{\eb}(\|\xi_0\|_{\E_b})e^{-\beta t}+Q_{\eb}(\|g\|_{L^2_b(\Omega)}),\quad t\geq 0,
\end{equation}
for some monotone increasing function $Q_\eb$ and $\beta>0$ independent of $u(t)$. Thus the proof is complete.  
\end{proof}
\begin{rem}
We note that the above proof becomes much less technically involved if in addition to dissipative assumption \eqref{f(s)s>-M} one also assumes extra condition 
\begin{equation}
\label{10}
f(u)u\geq aF(u)-K,\ F(u)=\int_0^uf(v)dv,
\end{equation}
for some positive constants $a$, $K$. In this case instead of \eqref{app.3} we have standard Gronwall estimate. However, as we have seen, assumption \eqref{10} can be omitted even in infinite-energy case. 
\end{rem}

\end{document}